\documentclass[hidelinks,onefignum,onetabnum]{siamart220329}

\usepackage{lipsum}
\usepackage{amsfonts}

\usepackage{epstopdf}
\usepackage{adjustbox}
\usepackage{algorithmic}
\ifpdf
  \DeclareGraphicsExtensions{.eps,.pdf,.png,.jpg}
\else
  \DeclareGraphicsExtensions{.eps}
\fi

\newsiamremark{remark}{Remark}
\newsiamremark{hypothesis}{Hypothesis}
\crefname{hypothesis}{Hypothesis}{Hypotheses}
\newsiamthm{claim}{Claim}
\newsiamthm{assum}{Assumption}

\headers{A PR-NF for stochastic dynamical systems}{M.~Yang, P.~Wang, D.~del-Castillo-Negrete, Y.~Cao, G.~Zhang}

\title{A pseudo-reversible normalizing flow for stochastic dynamical systems with various initial distributions
}

\author{M.~Yang\thanks{Fusion Energy Division, Oak Ridge National Laboratory,
 Oak Ridge, TN (\email{yangm@ornl.gov}, \email{delcastillod@ornl.gov}).}
 \and P.~Wang\thanks{Department of Mathematics, Auburn University, Auburn, AL (\email{pzw0030@auburn.edu}, \email{yzc0009@auburn.edu}).}
 \and D.~del-Castillo-Negrete\footnotemark[1]
\and Y.~Cao\footnotemark[2]
\and G.~Zhang\thanks{Computer Science and Mathematics Division, Oak Ridge National Laboratory, Oak Ridge, TN (\email{zhangg@ornl.gov}).}
}

\usepackage{amsopn}

\usepackage{bm}
\usepackage{subfigure}

\newlength{\tempdima}
\newcommand{\rowname}[1]
{\rotatebox{90}{\makebox[\tempdima][c]{\textbf{#1}}}}

\renewcommand{\thesubfigure}{\alph{subfigure}}
\newcommand{\mycaption}[1]
{\refstepcounter{subfigure}\textbf{(\thesubfigure) }{\ignorespaces #1}}

\ifpdf
\hypersetup{
  pdftitle={A pseudo-reversible normalizing flow for stochastic dynamical systems with various initial distributions},
  pdfauthor={M.~Yang, P.~Wang, D.~del-Castillo-Negrete, Y.~Cao, G.~Zhang}
}
\fi

\begin{document}

\maketitle

\begin{abstract}
We present a pseudo-reversible normalizing flow method for efficiently generating samples of the state of a stochastic differential equation (SDE) with different initial distributions. The primary objective is to construct an accurate and efficient sampler  that can be used as a surrogate model for computationally expensive numerical integration of SDE, such as those employed in particle simulation.  After training, the normalizing flow model can directly generate samples of the SDE's final state without simulating trajectories. Existing normalizing flows for SDEs depend on the initial distribution, meaning the model needs to be re-trained when the initial distribution changes. The main novelty of our normalizing flow model is that it can learn the conditional distribution of the state, i.e., the distribution of the final state conditional on any initial state, such that the model only needs to be trained once and the trained model can be used to handle various initial distributions. This feature can provide a significant computational saving in studies of how the final state varies with the initial distribution. We provide a rigorous convergence analysis of the pseudo-reversible normalizing flow model to the target probability density function in the Kullback–Leibler divergence metric. 
Numerical experiments are provided to demonstrate the effectiveness of the proposed normalizing flow model.
\end{abstract}


\section{Introduction}

The numerical solution of stochastic differential equations (SDE) for large ensembles of initial conditions is at the heart of modeling and simulation in science and engineering. For example, computations of interest to plasma physics rely on the solution of SDE tracking the evolution of electrons and ions in which the deterministic (drift) part describes the interaction with  
prescribed or self-consistent electromagnetic fields, and the stochastic part describes collisional effects. In fluid mechanics, 
SDE describe the evolution of passive tracers, e.g. pollutants in the oceans and the atmosphere, advected by prescribed velocity fields under the effect of molecular and/or turbulent diffusion. Further applications of SDE include the simulation of biological and chemical systems. From a mathematical perspective, SDE naturally appear as the discrete particle description formally equivalent to the Fokker-Planck description of the particle distribution function in the continuum limit. This equivalence is the foundation of particle-based Monte-Carlo methods for the numerical solution of Fokker-Planck type partial differential equations (PDE). Although particle based numerical methods offer advantages over continuum methods (e.g., finite difference, fine elements, or spectral), they also face limitations. In particular, to avoid numerical noise, Monte Carlo  methods need to integrate SDE for very large ensembles of particles to overcome the unfavorable slow convergence $\sim 1/\sqrt{N}$ where $N$ is the number of initial conditions.

To overcome numerical challenges faced by particle methods, in this paper we propose an efficient and accurate algorithm for the solution of SDE. Our approach is based 
on the use of pseudo-reversible normalizing flows (PR-NF) to efficiently generate samples of the state of an SDE for various initial distributions. Normalizing flows \cite{kobyzev2020normalizing,rezende2015variational,creswell2018generative,papamakarios2017masked, dinh2016density, grathwohl2018ffjord,GUO2022111202} are a class of models that map an unknown probability distribution to a standard probability distribution, e.g., the normal Gaussian distribution, such that one can generate samples from the unknown distribution by  sampling the standard distribution and then propagating through the flow model. 
%
%
%
Normalizing flows have been extended to learn stochastic dynamical systems  \cite{both2019temporal,lu2022learning,feng2021solving}. However, these models are not transferable, meaning they require model re-training every time the initial distribution is changed. 

We improve the normalizing flow models from three perspectives. First, we incorporate the initial state of the SDE as an input variable of the normalizing flow model, so that that our method aims at  learning the {conditional distribution} of the state at a given time instant. In this way, a single normalizing flow model can be used to handle different initial distributions by a simple convolution. In this sense, the proposed method can be viewed as an approach to learn the {Green's function} of the Fokker-Planck equation associated with the SDE. Second, we utilize the pseudo-reversible neural network architecture \cite{keller2021self, teng2021level, rippel2013high} to relax the strict reversibility constraint.
This architecture provides greater flexibility in designing flow transformations and offers potential benefits in terms of performance and applicability. The pseudo-reversible neural network architecture allows to use simple feed-forward neural networks to model both the forward and inverse flows, where an extra loss term is used to ensure pseudo-reversibility. Even though the proposed PR-NF has $\mathcal{O}(d^3)$ complexity in Jacobian determinant computation, we observe that it is not a bottleneck, especially with GPU-accelerated linear algebra libraries, for a wide range of physical processes defined in the six-dimensional phase space. Third, we provide a convergence analysis of the PR-NF model to the target probability density function in the Kullback–Leibler divergence metric, where the simple structure of the PR-NF model makes it feasible to exploit the universal approximation theorem of fully-connected neural networks in our convergence analysis.

The outline of the rest of the paper is as follows. In Section \ref{prob:set} we formulate the problem and define the SDE of interest.  The details of the surrogate model based on Normalizing Flows architecture are described in Section \ref{sec:prnn}. Section \ref{sec:analysis} provides the convergence analysis of the PR-NF model. Section \ref{sec:ex} is devoted to examples including a Brownian motion validation model, a problem of interest to magnetically confined fusion plasmas and an advection-diffusion  transport problem in fluid mechanics. 

\section{Problem setting}\label{prob:set}
%
We consider the $d$-dimensional stochastic process $X_t$, corresponding to the solution of the stochastic differential equation
\begin{equation}\label{eq_sde}
X_t= X_0 + \int_0^t {\bm b}(s, X_s) ds
+  \int_0^t{\bm \sigma}(s,  X_s) d  W_s \;\; \text{ with } X_0 \in \mathcal{D}, \,\, t\in[0,T],
\end{equation}
where $T>0$, $\mathcal{D}$ is a bounded domain in $\mathbb{R}^d$, ${\bm b}: [0,T] \times  \mathbb{R}^d \rightarrow \mathbb{R}^d$, ${\bm \sigma} :[0,T] \times \ \mathbb{R}^d \rightarrow \mathbb{R}^{d\times m}$ are drift and diffusion coefficients, respectively, and 
$W_s := (W_s^1, \ldots, W_s^m)^{\top}$ is a $m$-dimensional standard Brownian motion.
We assume that ${\bm b}$ and ${\bm \sigma}$ satisfy, for some constant $C$,
\begin{equation}
    |{\bm b}(t,{\bm x})| + |{\bm \sigma}(t,{\bm x})| \leq C(1+|{\bm x}|), \text{ for } {\bm x} \in \mathbb{R}^d, t\in [0,T],
\end{equation}
and  
\begin{equation}
   |{\bm b}(t,{\bm x}) - {\bm b}(t,{\bm y})| + |{\bm \sigma}(t,{\bm x}) - {\bm \sigma}(t,{\bm y})| \leq C|{\bm x}-{\bm y}|, \text{ for } {\bm x}, {\bm y} \in \mathbb{R}^d, t\in [0,T].
\end{equation}
For an initial state $X_0\in \mathcal{D}$, the SDE in Eq.~\eqref{eq_sde} has a unique time continuous solution $X_t$ \cite{oksendal2003stochastic}. 
As the state $X_t$ in Eq.~\eqref{eq_sde} depends on the initial condition $X_0$, we will write $X_t$ in the conditional form $X_t|X_0$ when we need to emphasize this dependence.

In applications one is typically interested in evaluating quantities of interest, ${\rm QoI}$, of the form \begin{equation}\label{eq:qoi}
{\rm QoI} \, = \int_{\mathbb{R}^d} F({\bm x_t})\, p_{X_t} ({\bm x_t}) d{\bm x}_t \, ,
\end{equation}
where $F$ denotes a physical property of the system, and $p_{X_t}$ is the probability density of the final state. Because of the dependence of the final state on the initial condition, normalizing flow strategies focusing on the distribution $p_{X_t} ({\bm x_t})$, e.g., Refs.~\cite{both2019temporal,lu2022learning,feng2021solving},
might not be efficient since every time the initial condition changes, the evaluation of ${\rm QoI}$ requires re-training of the neural network for  $p_{X_t} ({\bm x_t})$ corresponding to the new initial condition. 

To circumvent this problem we propose here a new method that focuses on the use of normalizing flows to create a surrogate model for the 
probability density $p_{X_t | X_0} ({\bm x_t|{\bm x_0}})$, 
which gives the transition probability of being at ${\bm x_t}$ stating at ${\bm x_0}$. Once a normalizing flow model is trained for $p_{X_t | X_0} ({\bm x_t|{\bm x_0}})$, the computation of the ${\rm QoI}$ reduces to the direct evaluation of the integral
\begin{equation}\label{eq:qoi1}
{\rm QoI} = \int_{\mathbb{R}^d} \int_{\mathbb{R}^d} F({\bm x_t})\, p_{X_t | X_0} ({\bm x_t}|{\bm x}_0) p_{X_0}({\bm x}_0)  d{\bm x_t} d{\bm x}_0 \, ,
\end{equation}
where $p_{X_0}({\bm x}_0)$ denotes the probability density of the initial condition. 
 In this approach the computational cost of evaluating ${\rm QoI}$ for different initial conditions reduces to the cost of performing a quadrature without a retraining overhead. 
If $\{{\bm x}_0^{(m)}\}$ for $i=1\, \ldots M$ is a set of $M$-samples drawn from $p_{X_0}({\bm x}_0)$, and $\{{\bm x}_t^{(n)}\}$ 
 is the corresponding set of $N$-samples drawn from $p_{X_t | X_0} ({\bm x_t}|{\bm x}_0)$, then the quadrature can be approximated as
\begin{equation}\label{eq:qoi2}
{\rm QoI} \approx \frac{1}{M N}\sum_{m=1}^M \sum_{n=1}^N F\left({\bm x}_t^{(n)}| \bm x_0^{(m)}\right),
\end{equation}

\section{The pseudo-reversible normalizing flow (PR-NF)}\label{sec:prnn}


A normalizing flow model can be viewed as a nonlinear transformation that maps an unknown target random variable, denoted by $X$, to a standard random variable, e.g., the standard Gaussian, denoted by $Z$. Because we intend to take into account the initial condition $X_0$ in our normalizing flow model, we define the target random variable $X$ as 
\begin{equation}\label{eq:y}
    X := ( X_0, X_t|X_0) \in \mathbb{R}^{2d}.
\end{equation}
A normalizing flow model includes a forward transformation mapping $X$ to $Z$ and an inverse transformation mapping $Z$ to $X$, i.e., 
\begin{equation}\label{eq:NF}
    Z = \bm h(X; \theta_h)\;\; \text{ and }\;\; \widehat{X} = \bm g(Z; \theta_g),
\end{equation}
where $\theta_h$, $\theta_g$ are parameters of the transformations and $\widehat{X} = X$ if $\bm g = \bm h^{-1}$. 
The function ${\bm h}$ moves (or {flows}) in the {normalizing} direction, i.e., from an unknown distribution to a standard distribution $p_{Z}$. The tradition of defining ${Z}$ as the standard normal random variable $\mathcal{N}(0,\mathbb{I}_{2d})$ gives rise to the name ``normalizing flow".
Let $p_{X}(\bm x)$ and $p_Z(\bm z)$ denote the PDF of $X$ and $Z$, respectively. 
%
The relation between the two PDFs can be obtained by the change of variables formula, i.e., 
\begin{equation}
\begin{aligned}
p_{X}({\bm x}) = p_{Z}({\bm z}) \left| \frac{\partial {\bm z}}{\partial {\bm x}} \right| &= p_{Z}({\bm h}({\bm x}))
|{\rm det} {\mathbf{J}_{h}}({\bm x})|,\\
\end{aligned}
\end{equation}
where ${\rm det} {\mathbf{J}_{\bm h}}({\bm x})$ is the determinant of the Jacobian matrix $\mathbf{J}_{\bm h}$ of the forward map $\bm h(\cdot)$.


\subsection{The PR-NF model's architecture}\label{sec:arch}
The PR-NF model uses the pseudo-reversible neural network architecture \cite{keller2021self, teng2021level, rippel2013high}, i.e., $\bm{g} \approx \bm{h}^{-1}$ in Eq.~\eqref{eq:NF}, to replace the exact reversible architectures used in most existing normalizing flow models, e.g., MAF \cite{papamakarios2017masked}, Real NVPs \cite{dinh2016density}. The typical exact reversible architectures rely on functions with either triangular Jacobians or Lipschitz continuity, where the Jacobian determinant can be efficiently approximated. However,
the relaxation of the reversibility condition increases the flexibility of the model design. Specifically, we define $\bm{h}(\cdot; \theta_h)$ and $\bm{g}(\cdot; \theta_g)$ as two independent fully-connected neural networks, and the pseudo-reversibility is imposed by adding a soft constraint $\|\widehat{X} - X\|_2^2$ to the loss function. This architecture is different from autoencoders because the size of the bottleneck, i.e., the dimension of $Z$, is the same as the dimension of $X$ and $\widehat{X}$. Figure~\ref{fig_nf} illustrates the performance of the PR-NF on a standard two-dimensional dataset. We observe that the left panel (i.e., the true distribution) and the right panel (i.e., the PR-NF generated distribution) have a good agreement, meaning the relaxation of the reversibility does not cause a significant error. 
\begin{figure}[h!]
    \centering
  {\includegraphics[width=0.9\textwidth]{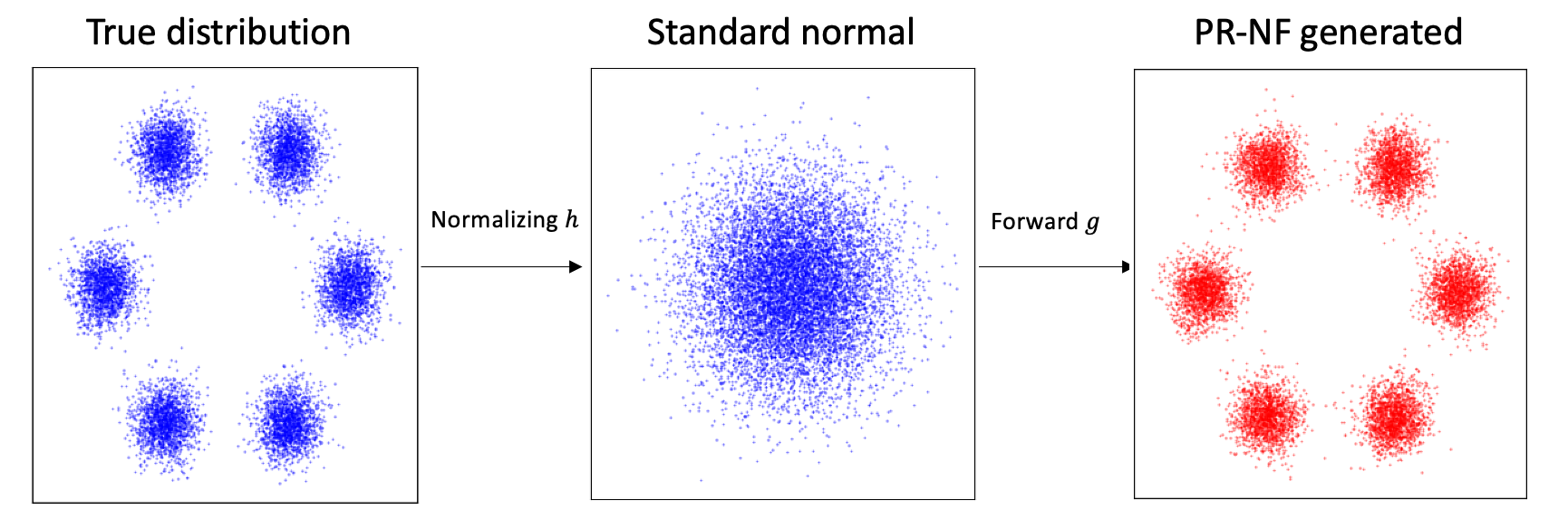}}
  \vspace{-0.15cm}
    \caption{Illustration of the PR-NF's performance. The function ${\bm h}$ maps the true distribution (left) to the standard normal distribution (middle), and the function $\bm g$ generates new samples of the true distribution. We observe that the left panel and the right panel have a good agreement, meaning the relaxation of the reversibility does not cause a significant error.}
    \label{fig_nf}
    \vspace{-0.3cm}
\end{figure}

The novelty of the PR-NF model is to learn the conditional distribution in Eq.~\eqref{eq:qoi1}, such that we can use the trained PR-NF model to compute the QoI in Eq.~\eqref{eq:qoi1} for arbitrary initial conditions, $p_{X_0}(\bm x_0)$, without re-training. The key ingredient is to include the initial state $X_0$ into the input of the PR-NF model, i.e., Eq.~\eqref{eq:y}. The network architecture is shown in Figure \ref{fig_nf_x0}. Specifically, the mapping $\bm h(\cdot)$ consists of two components, i.e., $\bm h(\bm x_0, \bm x_t|\bm x_0) = (\bm h_0(\bm x_0), \bm h_1(\bm x_0, \bm x_t|\bm x_0; \theta_h))$, where $\bm h_1(\bm x_0, \bm x_t|\bm x_0; \theta_h): \mathbb{R}^{2d} \mapsto \mathbb{R}^d$ is a fully connected neural network with tuning parameter $\theta_h$, and $\bm h_0(\bm x_0): \mathbb{R}^d \mapsto \mathbb{R}^d$ is an identity mapping depending only on $\bm x_0$. We emphasize that it is critical to include $\bm x_0$ as one input of $\bm h_1$ to characterize the conditional distribution $p_{X_t|X_0}(\bm x_t|\bm x_0)$. We denote by $\bm z_0$ the output of $\bm h_0(\bm x_0)$ and $\bm z_t$ the output of $\bm h_1(\bm x_0, \bm x_t|\bm x_0; \theta_h)$. 
The inverse map $\bm g$ has a similar structure as the forward map, i.e., $\bm g(\bm z_0, \bm z_t) = (\bm g_0(\bm z_0), \bm g_1(\bm z_0, \bm z_t; \theta_g))$, where $\bm g_0$ is an identity map and $\bm g_1$ is a fully connected neural network with tuning parameter $\theta_g$. We denote by $\widehat{\bm x}_0$ and $\widehat{\bm x}_t$ the outputs of $\bm g_0$ and $\bm g_1$, respectively. Note that $\widehat{\bm x}_0 = \bm z_0 = \bm x_0$ because $\bm h_0$ and $\bm g_0$ are identity maps. 
%
%
\begin{figure}[h!]
    \centering
  {\includegraphics[width=0.7\textwidth]{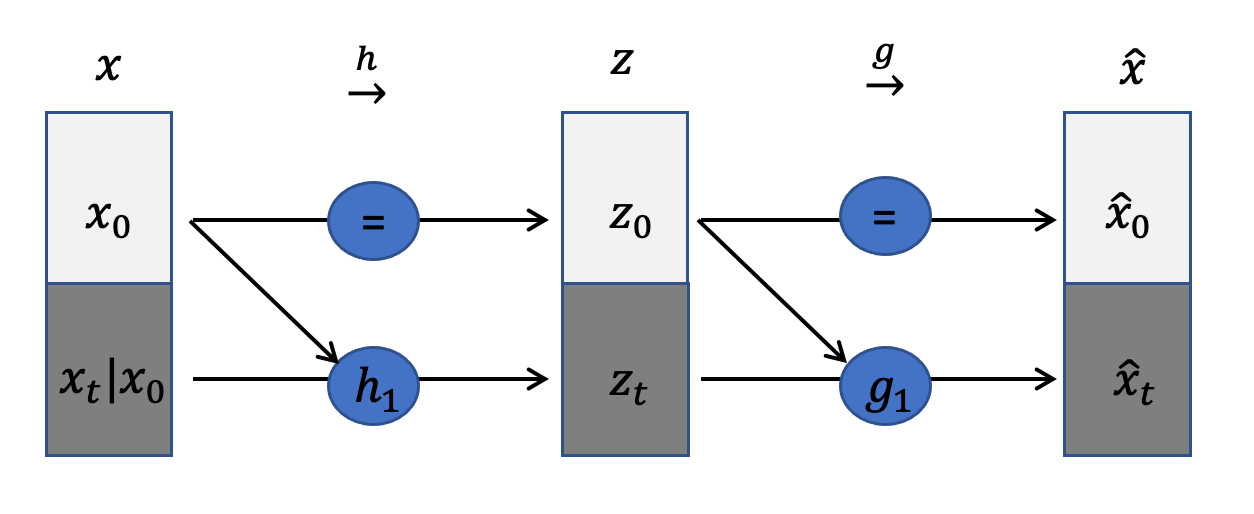}}
  \vspace{-0.3cm}
    \caption{The network architecture of the proposed PR-NF model.}
    \label{fig_nf_x0}
    \vspace{-0.3cm}
\end{figure}

\subsection{The loss function}\label{sec:loss}
The loss function is defined based on a training data set, i.e.,
\begin{equation}\label{eq:train}
\mathcal{V}_{\rm train} = \left\{ {\bm x}^{(n)}\right\}_{n=1}^{N} = \left\{\bm x_0^{(n)},  \bm x_t^{(n)}|\bm x_0^{(n)}\right\}_{n=1}^{N},
\end{equation}
where $\{\bm x_0^{(n)}\}_{n=1}^N$ is generated from the uniform distribution over the bounded domain $\mathcal{D}$ in Eq.~\eqref{eq_sde}. For each initial state $\bm x_0^{(n)}$, we numerically solve the SDE in Eq.~\eqref{eq_sde} to obtain the sample $\bm x_t^{(n)}|\bm x_0^{(n)}$. The loss function consists of two components, i.e., 
\begin{equation}\label{eq_loss}
\mathcal{L} = \mathcal{L}_1 + \lambda \mathcal{L}_2,
\end{equation}
where $\mathcal{L}_1$ is the negative log-likelihood loss defined by
\begin{equation}\label{eq:loglike}
 \mathcal{L}_1 = - \frac{1}{N}\sum_{n=1}^{N} \left({\rm log}p_{Z_t}({\bm h}_1({\bm x}^{(n)}; \theta_{h})) +  {\rm log} \left|{\rm det} {\bf J}_{\bm h} ({\bm x}^{(n)}; \theta_{h})\right|\right),
\end{equation}
with $p_{Z_t}(\cdot)$ the probability density function of the standard normal distribution, and $\mathcal{L}_2$ is the pseudo-reversibility loss that measures the difference between ${\bm x}$ and $\widehat{\bm x}$, i.e.,
\begin{equation}
    \mathcal{L}_2 = \frac{1}{N} \sum_{n=1}^{N} \left( \left\|{\bm x}^{(n)} - {\bm g} ( {\bm h}({\bm x}^{(n)} ; \theta_{h}); \theta_{g})\right\|_2^2  + \left|\det {\bf J}_{\bm g}({\bm h}({\bm x}^{(n)})) \det {\bf J}_{\bm h}({\bm x}^{(n)}) - 1 \right| \right),
\end{equation}
and $\lambda$ is a hyperparameter that will be discussed in Section \ref{sec:hyper}. 

The Jacobian determinant in $\mathcal{L}_1$ can be written as  
\begin{equation}\label{eq_Jac_reduce}
    \left|{\rm det} {\mathbf{J}_{\bm h}}({\bm x})\right| = \left|\text{det}\begin{pmatrix}
\mathbb{I}_d, \quad  0\\
\frac{\partial \bm z_t}{\partial \bm x_0}, \frac{\partial \bm z_t}{\partial \bm x_t}
\end{pmatrix} \right| = \left|\text{det}\left(\frac{\partial \bm z_t}{\partial \bm x_t}\right)\right|,
\end{equation}
where $\mathbb{I}_d$ is a $d\times d$ identity matrix. 
The simplification in Eq.~\eqref{eq_Jac_reduce} is based on the fact ${\bm h_0}$ is an identity map with respect to $\bm x_0$. Thus, although we doubled the size of the input of the PR-NF model, the size of the Jacobian matrix is still $d \times d$. The proposed PR-NF has $\mathcal{O}(d^3)$ complexity in Jacobian determinant computation. However, we observe that it is not a bottleneck with GPU-accelerated linear algebra libraries in Pytorch and TensorFlow, 
especially for a wide range of physical processes, i.e., the examples in Section \ref{sec:ex}, in the phase space with dimension up to six. 




\subsection{Hyperparameter tuning}\label{sec:hyper}

The hyperparameter $\lambda$ in Eq.~\eqref{eq_loss} determines the importance of the reversibility loss $\mathcal{L}_2$. Thus, a hyperparameter tuning is needed to achieve a good performance. In this work, we will perform grid search and use the cross-entropy as a metric to find the best value for $\lambda$. Specifically, we will generate a set of candidates, denoted by $\{\lambda_j\}_{j=1}^J$. For each $\lambda_j$, we train one PR-NF model, denoted by $(\bm h^{(j)}(\bm x), \bm g^{(j)}(\bm z))$. Then, we generate $M$ samples of 
$\widehat{\bm x}$ by running the inverse mapping, i.e., 
\begin{equation}\label{eq:sample}
\left\{ \widehat{\bm x}^{(j,m)} = \bm g^{(j)}(\bm z^{(m)}), m = 1, \ldots, M \right\}, 
\end{equation}
where $\{\bm z^{(m)}\}_{m=1}^M = \{\bm z_0^{(m)}, \bm z_t^{(m)}\}_{m=1}^M$ with $\bm z_0^{(m)}$ generated uniformly from $\mathcal{D}$ and $\bm z_t^{(m)}$ generated from the standard normal distribution. Because it does not require solving the SDE in Eq.~\eqref{eq_sde}, we can generate a large number of samples. Next, we build a kernel density estimator using the samples in Eq.~\eqref{eq:sample}, and compute the cross-entropy by
\begin{equation}\label{eq:cross}
    H(\lambda_j) = - \frac{1}{N} \sum_{n=1}^N \log(p_{\rm KDE}( \bm x_t^{(n)})),
\end{equation}
where $p_{\rm KDE}$ is the kernel density estimator based on the samples in Eq.~\eqref{eq:sample} and $\{\bm x_t^{(n)}\}_{n=1}^N$ is the training data set in Eq.~\eqref{eq:train}. The optimal hyperparameter $\lambda$ is obtained by minimizing the cross-entropy, i.e., $\lambda = \text{argmin}\, H(\lambda_j)$.

\subsection{Discussion on the computational cost}
The algorithmic cost of sampling methods consists of two parts, the ``offline cost" and the ``online cost''. The offline cost of an algorithm represents the running time for the training process. It is determined by the architecture of the neural network, the optimizer (numerical scheme), the size of the training dataset and the number of training epochs (``time" steps). 
An epoch in machine learning means one complete pass of the training dataset through the algorithm. We choose the number of epochs such that the loss function stably attains the minimum.
The Monte Carlo method for sampling has no training process. However, it needs to be run repeatedly for each new initial condition. The PR-NF model is a generative model where the offline cost accounts for most of the 
 total algorithmic cost. However, the offline cost is a one-time cost (no additional training is needed for future simulations). We denote the offline cost of the PR-NF method as $C_{\rm offline}$. It is easy to utilize the GPU architecture to accelerate the computation for the training process under the PyTorch machine learning framework.

The online cost is the computational cost for sampling. Although the Monte Carlo method is conceptually simple and algorithmically straightforward, the associated computational cost  can be staggeringly high. In general, the method requires many samples to get a good approximation, which may incur an arbitrarily large total runtime if the processing time of a single sample is high \cite{shonkwiler2009explorations}. In contrast, the PR-NF model is a transformation that maps samples from the standard normal distribution to the distribution of interest. For an initial distribution, the computational cost of the MC method is significantly greater than the online cost of the PR-NF method because it involves the sampling and numerical integration of the SDE in Eq.~\eqref{eq_sde}. The computational cost depends on how many samples need to generate. We denote $C_{\rm MC}$ and $C_{\rm PR-NF}$ as the computational cost of an initial distribution for MC and PR-NF methods, respectively. We have $C_{\rm MC} \gg C_{\rm PR-NF}$. 
To test $N_{\rm init}$ different initial distributions, the cost of the PR-NF method $\mathcal{O}(C_{\rm offline} + C_{\rm PR-NF} \times N_{\rm init})$ is cheaper than the MC method $\mathcal{O}(C_{\rm MC}  \times N_{\rm init})$ when $N_{\rm init}$ is large.

\section{Convergence analysis of the PR-NF model}\label{sec:analysis}
This section provides the convergence analysis of the pseudo-reversible normalizing flow (PR-NF) model in Section \ref{sec:prnn}. For notational simplicity, the target random variable and its probability density function are denoted by $X$ and $p_X({\bm x})$, respectively. Here, we limit our attention to single-hidden-layer neural networks. The work in \cite{trans} shows that for any pair of well-behaved\footnote {A well-behaved distribution $p_X({\bm x})$ means that $p_X({\bm x})>0$ for all ${\bm x}\in \mathbb{R}^d$, and all conditional probabilities ${\rm Pr}(X_i \le x_i | {\bm x}_{<i})$ are differentiable, for $i = 1, 2, \ldots, d$.} distributions $p_X({\bm x})$ and $p_Z({\bm z})$, there exists a diffeomorphism $\bm z = \bm f(\bm x) = (f_1(\bm x), \ldots, f_d(\bm x))$ that can transform $p_X({\bm x})$ to $p_Z({\bm z})$, i.e., 
\begin{equation}\label{eq_denp}
    p_{X}(\bm x) = p_Z(\bm f(\bm x))|\det {\bf J}_{\bm f}(\bm x)|.
\end{equation}
The proposed PR-NF model can be viewed as an approximation of the diffeomorphism, where $\bm h(\bm x)$ and $\bm g(\bm z)$ in Eq.~\eqref{eq:NF} approximate $\bm f$ and $\bm f^{-1}$, respectively. 

We study the convergence of the PR-NF model in three stages. Section \ref{sec:exist} shows the existence of a convergent neural network approximations to the target diffeomorphism $\bm f$ by exploiting the universal approximation theorm of fully connected neural networks. Section \ref{sec:l1l2} shows that the convergence of the loss functions $\mathcal{L}_1$ and $\mathcal{L}_2$ in Eq.~\eqref{eq_loss} is a necessary condition to obtain the convergent approximation discussed in Section \ref{sec:exist}, which verifies that appropriateness of the  loss function used to train the PR-NF model. Section \ref{sec:kl} shows that the convergence of $\mathcal{L}_1$ and $\mathcal{L}_2$ plus some mild assumptions on the $\bm h$ and $\bm g$ is sufficient to ensure the convergence of the KL divergence between the true and the approximate probability density function.



\subsection{Existence of a convergent approximation to the diffeomorphism}\label{sec:exist}
%
%
\begin{assum}\label{ass1}
We assume that the determinant of the Jacobian matrix and any partial derivative of $\bm f \in C^1(\mathbb{R}^d) $ are bounded, i.e., 
 \begin{equation}\label{assum_1}
     \left|\frac{\partial f_i(\bm x)}{\partial x_j}\right| < A_1 < \infty,
    \quad |\det \mathbf{J}_{\bm f}(\bm x)| > A_2 > 0,
 \end{equation}
where $A_1$ and $A_2$ are positive constants and $i,\ j \in \{1,\ldots,d\}$. 
Additionally, the probability density function $p_X(\bm x)$ is bounded and all the conditional 
probabilities are differentiable. There exists a positive constant $K$ such that  
 \begin{equation}\label{assum_3}
     p_X(\bm x) < A_3\exp(- \alpha\|\bm x\|^2)\; \text{ for }\; \|\bm x\| > K,
 \end{equation}
 where $\alpha>0$ and $A_3>0$ are positive constants and $\|\bm x\|$ denotes the $l^2$ norm.
\end{assum}

Under Assumption \ref{ass1}, Lemma \ref{UAT} and Lemma \ref{lem_inv} demonstrate the existence and reversibility of convergent flows $\bm h$ and $\bm g$, respectively. To proceed, we introduce some multivariate notations: $\mathbb{Z}_+^d$ denote
the lattice of nonnegative multi-integers in  $\mathbb{R}^d$. 
  For $\mathbf{m} = (m_1,\ldots,m_d) \in \mathbb{Z}_+^d$, we set $|\mathbf{m}|=m_1+\ldots+m_d$, $\bm{x}^\mathbf{m} = x_1^{m_1}\cdots x_d^{m_d}$, and
    $D^\mathbf{m} = \frac{\partial^{|\mathbf{m}|}}{\partial x_1^{m_1}\cdots \partial x_d^{m_d}}$. We also define
\[
C^k(\Omega) := \big\{{\bm f} : D^\mathbf{m}f_i \in C(\Omega)\ \text{for all}\ i\in \{1,\ldots,d\}, \, \text{and} \, |\mathbf{m}| \le k\big\}.
\]

\begin{lemma}[Theorem 4.1 in \cite{UAT}]\label{UAT}
For any given $\varepsilon>0$, there exists a single-hidden-layer neural network $\bm h = (h_1, \ldots, h_d)$ with sufficient number of neurons to approximate a function $\bm f\in C^k(\mathbb{R}^d)$ in a compact set $\Omega\subset \mathbb{R}^d$
such that 
\begin{equation}\label{eq_lemmaUAT}
    \max_{{\bm x}\in\Omega} |D^\mathbf{m} f_i(\bm x)-D^\mathbf{m} h_i (\bm x)|<\varepsilon,
\end{equation}
for all $i\in \{1,\ldots,d\}$ and $\mathbf{m}\in \mathbb{Z}_+^d$ with $|\mathbf{m}|\le k$.
\end{lemma}

\begin{lemma}\label{lem_inv}
Under Assumption \ref{ass1} and Lemma \ref{UAT}, the flow $\bm h$ approximating $\bm f \in C^k(\mathbb{R}^d)$ is invertible for sufficiently small $\varepsilon > 0$.
%
\end{lemma}
\begin{proof}
Let $\mathcal{S}_d$ be the set of all permutation of $\{1,\ldots,d\}$, we have
\begin{equation}\label{thm1_jf-jh}
|\det {\bf J}_{\bm f}(\bm x) - \det {\bf J}_{\bm h}(\bm x)|
=\left| \sum_{\sigma\in \mathcal{S}_d}\left( {\rm sgn}(\sigma) \left(\prod_{i = 1}^d \frac{\partial f_i}{\partial x_{\sigma_i}} -\prod_{i = 1}^d \frac{\partial h_i}{\partial x_{\sigma_i}} \right)\right) \right|
< C \varepsilon
\end{equation}
for any $\bm x \in \Omega$, where the last inequality holds by Lemma \ref{UAT} with $k = 1$. According to Eq.~\eqref{assum_1}, we have
\begin{equation}\label{thm1_jf}
    A_2 < |\det {\bf J}_{\bm f}(\bm x)| = \left| \sum_{\sigma\in \mathcal{S}_d}\left( {\rm sgn}(\sigma)\prod_{i = 1}^d \frac{\partial f_i}{\partial x_{\sigma_i}}\right)\right|<CA_1^d.
\end{equation}  
We choose $\varepsilon<\frac{A_2}{2C}$ such that 
\begin{align}
    &|\det {\bf J}_{\bm h}(\bm x)| \le |\det {\bf J}_{\bm f}(\bm x)|+|\det {\bf J}_{\bm f}(\bm x) - \det {\bf J}_{\bm h}(\bm x)|<\frac{A_2}{2}+CA_1^d\label{thm1_jhup}, \\
    &|\det {\bf J}_{\bm h}(\bm x)| \ge |\det {\bf J}_{\bm f}(\bm x)|-|\det {\bf J}_{\bm f}(\bm x) - \det {\bf J}_{\bm h}(\bm x)|>\frac{A_2}{2}>0\label{thm1_jhlow}.
\end{align}
By the inverse function theorem, $\bm h$ is invertible in $\Omega$. 
\end{proof}

Similarly, we can prove the  invertibility of ${\bm g}$ for the function ${\bm f}^{-1}$ with Eq.~\eqref{eq_lemmaUAT}. So far, we proved the existence of convergent neural network $\bm h$ and $\bm g$ to approximate $\bm f$ and $\bm f^{-1}$, respectively.

\subsection{Convergence of the loss function}\label{sec:l1l2}
We now prove that 
the convergence of the loss functions $\mathcal{L}_1$ and $\mathcal{L}_2$ in Eq.~\eqref{eq_loss} is a necessary condition to obtain the convergent approximation discussed in Section \ref{sec:exist}. To proceed, we define two auxiliary 
random variables 
\begin{equation}\label{eq:xxx}
\widetilde{X} = {\bm h^{-1}}(Z)\; \text{ and } \; \widehat{X} = {\bm g}(Z),
\end{equation}
where $Z$ follows the standard normal distribution. 
Using the change of variables formula, the probability density functions of  $\widetilde{X}$ and $\widehat{X}$ are defined by
\begin{equation}\label{p_xtilde}
    p_{\widetilde X}(\bm x) = p_Z(\bm h (\bm x))|\det {\bf J}_{\bm h}(\bm x)|, \quad p_{\widehat X}(\bm x) = p_Z(\bm g^{-1} (\bm x))|\det {\bf J}_{\bm g^{-1}}(\bm x)|.
\end{equation}

For simplicity, we consider the loss functions in the continuous form, i.e., 
\begin{equation}\label{con_l1}
\mathcal{L}_1 = -\int_{\mathbb{R}^d} p_X(\bm x) \log p_{\widetilde{X}}(\bm x) d\bm x,
\end{equation}
and 
\begin{equation}\label{con_l2}
\mathcal{L}_2 = \int_{\Omega} p_X(\bm x)\left(\|{\bm g}({\bm h}(\bm x))-\bm x\|^2 + |\det {\bf J}_{\bm g}({\bm h}(\bm x)) \det {\bf J}_{\bm h}(\bm x) - 1|\right)d\bm x,
\end{equation}
where $\Omega$ is a compact set. For the variables ${\widetilde X}$ and ${\widehat X}$, we have the following assumptions.
%



\begin{assum}\label{ass2}
We assume the density functions $p_{X}(\bm x)$ and $p_{\widetilde X}(\bm x)$ satisfy

 \begin{equation}\label{assum_4}
     A_4\exp(-\alpha \|\bm x\|^2) < \frac{p_X(\bm x)}{p_{\widetilde{X}}(\bm x)} < A_4\exp(\alpha \|\bm x\|^2),
 \end{equation}  
 where $\alpha$ and $A_4$ are positive constants. 
\end{assum}
Lemma \ref{Thm1} and Theorems \ref{Thm2}, \ref{Thm3} below will demonstrate that both $\mathcal{L}_1$ and $\mathcal{L}_2$ will approach the optimal constants in this case, with a difference between the loss and the optimal constant of the order  $O(\varepsilon)$. 
\begin{lemma}\label{Thm1}
For arbitrarily small $\varepsilon>0$, there exists a random variable $\widetilde{X}$  with the density function $p_{\widetilde X}(\bm x)$ defined in Eq.~\eqref{p_xtilde} such that
\begin{equation}\label{thm1_1}
|p_{X}(\bm x)- p_{\widetilde X}(\bm x)| < C \varepsilon, \,\,\text{ for }\,\,  \bm x \in \Omega,
\end{equation}
 where  $\Omega\subseteq\mathbb{R}^d$ is a compact set and $C>0$ is a constant.
\end{lemma}

\begin{proof}

By Lemma \ref{UAT} with $k = 0$, for any $\bm x\in\Omega$ we have 
\begin{equation}\label{thm1_f-h}
    \|\bm f(\bm x)-\bm h (\bm x)\| = \sqrt{\sum_{i=1}^d |f_i(\bm x)-h_i (\bm x)|^2}<C\varepsilon.
\end{equation}
%
The standard normal random variable $Z \in \mathbb{R}^d$ has bounded gradient, i.e., $\|\nabla p_Z(\bm z)\|< C$ for any $\bm z\in\mathbb{R}^d$. Applying the mean value theorem, we have 
\begin{equation}\label{thm1_pzf-pzh}
        |p_Z(\bm f(\bm x)) - p_Z(\bm h(\bm x))|
        =|\nabla p_Z(\xi)\cdot(\bm f(\bm x)- \bm h(\bm x))|
        \le\|\nabla p_Z(\xi)\| \|\bm f(\bm x)- \bm h(\bm x)\|
        <C \varepsilon,
\end{equation}
where $\xi$ is between $\bm f(\bm x)$ and $\bm h(\bm x)$, $\varepsilon$ is given in Lemma \ref{UAT} and $C>0$ is a constant.
Because $Z$ has bounded density function, we can exploit Eqs.~\eqref{thm1_jf-jh}, \eqref{thm1_jhup}, \eqref{thm1_pzf-pzh} to derive 
\begin{equation}\label{thm1_8}
    \begin{aligned}
        & \left|p_{X}(\bm x)- p_{\widetilde X}(\bm x)\right|\\
        =&\, \left|p_Z(\bm f(\bm x))|\det {\bf J}_{\bm f}(\bm x)| - p_Z(\bm h (\bm x))|\det {\bf J}_{\bm h}(\bm x)| \right|\\
        =& \,\left|p_Z(\bm f(\bm x))(|\det {\bf J}_{\bm f}(\bm x)| - |\det {\bf J}_{\bm h}(\bm x)|)+|\det {\bf J}_{\bm h}(\bm x)|(p_Z(\bm f(\bm x)) - p_Z(\bm h(\bm x)))\right|\\
        \le& \, \left| p_Z(\bm f(\bm x))\right| \left| \det {\bf J}_{\bm f}(\bm x) - \det {\bf J}_{\bm h}(\bm x)\right| + \left| \det {\bf J}_{\bm h}(\bm x)\right|\left| p_Z(\bm f(\bm x)) - p_Z(\bm h(\bm x))\right|< C \varepsilon
    \end{aligned}  
\end{equation}
for any $\bm x\in \Omega$ which concludes the proof.
\end{proof}

Based on the results from Lemma \ref{Thm1}, we have the convergence of the continuous loss function $\mathcal{L}_1$ defined in Eq.~\eqref{con_l1}.
\begin{theorem}\label{Thm2}
For arbitrarily small $\varepsilon>0$, there exists a random variable $\widetilde{X}$ with the density function $p_{\widetilde X}(\bm x)$ in Eq.~\eqref{p_xtilde} such that $\mathcal{L}_1$ has the bound
\begin{equation}\label{thm2_bound}
    -\int_{\mathbb{R}^d} p_X(\bm x) \log p_X(\bm x) dx \le \mathcal{L}_1 <-\int_{\mathbb{R}^d} p_X(\bm x) \log p_X(\bm x) d\bm x +\varepsilon.
\end{equation}
\end{theorem}

\begin{proof}
Using the inequality $\log t \le t-1$ for all $t>0$, we have
\begin{equation}\label{thm2_lb}
    \begin{aligned}
        -\int_{\mathbb{R}^d} p_X(\bm x) \log p_X(\bm x) d\bm x - \mathcal{L}_1
        =&\int_{\mathbb{R}^d} p_X(\bm x) \log \frac{p_{\widetilde X}(\bm x)}{P_X(\bm x)} d\bm x\\
        \le & \int_{\mathbb{R}^d} p_X(\bm x) \left(\frac{p_{\widetilde X}(\bm x)}{P_X(\bm x)} - 1\right) d\bm x\\
        =&\int_{\mathbb{R}^d} p_{\widetilde X}(\bm x) - p_X(\bm x) d\bm x = 0.        
    \end{aligned}
\end{equation}
To prove the right side of Eq.~\eqref{thm2_bound}, we need to show that
\begin{equation}
    \int_{\mathbb{R}^d} p_X(\bm x) \log \frac{p_X(\bm x)}{p_{\widetilde{X}}(\bm x)} d\bm x < \varepsilon.
\end{equation}
We define $B_d(r) := \{\bm x\in\mathbb{R}^d: \|\bm x\|\le r\}$ and $V_{d-1}(r)$ be the volume of $\{\bm x\in\mathbb{R}^d: \|\bm x\| = r\}$. For every $\varepsilon_1 < \min\{A_3\exp(-\alpha),  A_3\exp(- \alpha K^2)\}$,  $K$ is the constant in Assumption \ref{ass1}, we let $K_1 = \sqrt{-\frac{1}{\alpha} \log\frac{\varepsilon_1}{A_3}}$,
then $K_1 > \max\{1, K\}$. According to Eq.~\eqref{assum_3} in Assumption \ref{ass1}, $p_X(\bm x) < A_3\exp(- \alpha\|\bm x\|^2)\le\varepsilon_1$ over the domain $B_d(K_1)^c$.
Define $\Omega_{\varepsilon_1} := \{\bm x: p_X(\bm x)\ge\varepsilon_1\}$, then $\Omega_{\varepsilon_1}$ is contained in $B_d(K_1)$. Since $p_X(\bm x)$ is continuous, $\Omega_{\varepsilon_1}$ is closed. Moreover, as a bounded closed set in $\mathbb{R}^d$, $\Omega_{\varepsilon_1}$ is compact.
The KL divergence between distributions  $p_X$ and $p_{\widetilde{X}}$ has the following bound
\begin{equation}\label{thm2_kl}
\begin{aligned}
D_{\rm KL}(p_X \, \| \, p_{\widetilde{X}})
=& \int_{\mathbb{R}^d}p_X(\bm x)\log \frac{p_X(\bm x)}{p_{\widetilde{X}}(\bm x)} d\bm x \\
\le& \int_{p_X(\bm x) \ge \varepsilon_1} p_X(\bm x)\log \frac{p_X(\bm x)}{p_{\widetilde{X}}(\bm x)} d\bm x \\
& + \int_{\{p_X(\bm x) < \varepsilon_1\} \cap B_d(K_1)}p_X(\bm x) \log \frac{p_X(\bm x)}{p_{\widetilde{X}}(\bm x)} d\bm x \\
& + \int_{\{p_X(\bm x) < \varepsilon_1\}\cap B_d(K_1)^c} p_X(\bm x) \log \frac{p_X(\bm x)}{p_{\widetilde{X}}(\bm x)} d\bm x \\
:=& I_1 + I_2 + I_3.
\end{aligned}  
\end{equation}

Next we investigate all terms on the right side of Eq.~\eqref{thm2_kl}.
By Lemma \ref{Thm1}, for every $\varepsilon_2 > 0$, there exists variable $\widetilde{X}$ such that
\begin{equation}
    |p_{X}(\bm x)- p_{\widetilde X}(\bm x)| < C \varepsilon_2,
\end{equation}
for any $\bm x \in \Omega_{\varepsilon_1}$.
Choosing $\varepsilon_2 < \frac{\varepsilon_1}{2C}$, we have $p_{\widetilde{X}}(\bm x) > \varepsilon_1/2$ for the case $p_X(\bm x) \ge \varepsilon_1$. Applying the mean value theorem, $I_1$ has the bound 
\begin{equation}\label{thm3_2}
\begin{aligned}
I_1  
= \, &\int_{p_X(\bm x) \ge \varepsilon_1} p_X(\bm x)\left( \log (p_X(\bm x))-\log(p_{\widetilde{X}}(\bm x)) \right)d\bm x\\
\le \, &\int_{p_X(\bm x) \ge \varepsilon_1} p_X(\bm x) \left|\frac{1}{{\xi}}\right| \left|p_X(\bm x) - p_{\widetilde{X}}(\bm x)\right|d\bm x\\
\le \, &\frac{2C\varepsilon_2}{\varepsilon_1} \int_{p_X(\bm x) \ge \varepsilon_1} p_X(\bm x)d\bm x
\le \, \frac{2C\varepsilon_2}{\varepsilon_1},
\end{aligned}  
\end{equation}
where $\left|\frac{1}{\xi}\right| \le \frac{2}{\varepsilon_1}$ because  $\xi$ is between $p_X(\bm x)$ and $p_{\widetilde{X}}(\bm x)$. Since we define  $\varepsilon_2 < \frac{\varepsilon_1}{2C_1}$, we are able to let  $\varepsilon_2/\varepsilon_1 \rightarrow 0$ to guarantee $I_1 \rightarrow 0$.

Positive density functions $p_X(\bm x)$, $p_{\widetilde{X}}(\bm x)$ are continuous, thus $\log \frac{p_X(\bm x)}{p_{\widetilde{X}}(\bm x)}$ is continuous and bounded over $B_d(K)$, i.e., $\log \frac{p_X(\bm x)}{p_{\widetilde{X}}(\bm x)} < C_1$ for ${\bm x} \in B_d(K)$. By Eq.~\eqref{assum_4} in Assumption \ref{ass1}, it follows that $\log \frac{p_X(\bm x)}{p_{\widetilde{X}}(\bm x)} < C_2K_1^2 $ for $\bm x\in B_d(K_1)\setminus B_d(K)$. Hence, $\log \frac{p_X(\bm x)}{p_{\widetilde{X}}(\bm x)} < \max \{C_1,\ C_2K_1^2\}$ for $\bm x \in B_d(K_1)$.
The term $I_2$ has the bound
\begin{equation}\label{thm3_3}
\begin{aligned}
I_2 & = \int_{\{p_X(\bm x)<\varepsilon_1\}\cap B_d(K_1)} p_X(\bm x)\log \frac{p_X(\bm x)}{p_{\widetilde{X}}(\bm x)}d \bm x\\
& \le \int_{\{p_X(\bm x)<\varepsilon_1\}\cap B_d(K_1)}\left|p_X(\bm x)\log \frac{p_X(\bm x)}{p_{\widetilde{X}}(\bm x)}\right| d\bm x\\
& \le C\varepsilon_1\int_{B_d(K_1)} \max\{1,\ K_1^2\}d \bm x \\
&= C\varepsilon_1\frac{\pi^{d/2}}{\Gamma(d/2+1)}\max \{K_1^d,\ K_1^{d+2}\}\\
&\le C \max\{\varepsilon_1 (\log\frac{\varepsilon_1}{A_1})^{d/2},\ \varepsilon_1 (\log\frac{\varepsilon_1}{A_1})^{d/2+1}\},
\end{aligned}  
\end{equation}
where $C$ is a constant and $I_2 \rightarrow 0$ as $\varepsilon_1 \rightarrow 0$. The last inequality holds due to  $K_1 = \sqrt{-\frac{1}{\alpha} \log\frac{\varepsilon_1}{A_1}}$ . Lastly we have the bound for $I_3$

\begin{equation}\label{thm3_4}
\begin{aligned}
I_3 & = \int_{\{p_X(\bm x)<\varepsilon_1\}\cap B_d(K_1)^c} p_X(\bm x)\log \frac{p_X(\bm x)}{p_{\widetilde{X}}(\bm x)} d\bm x\\
& = \int_{B_d(K_1)^c}p_X(\bm x)\log \frac{p_X(\bm x)}{p_{\widetilde{X}}(\bm x)} d\bm x\\
& \le C\int_{B_d(K_1)^c} \exp(-\alpha\|\bm x\|^2) \|\bm x\|^2 d\bm x.
\end{aligned}
\end{equation}
According to the identity $\int_{B_d(R_2)\setminus B_d(R_1)}y(\|\bm x\|)d\bm x = V_{d-1}(1)\int_{R_1}^{R_2}y(r)r^{d-1}dr$, we have
\begin{equation}\label{thm3_5}
\begin{aligned}
I_3& \le C V_{d-1}(1)\int_{K_1}^{+\infty} \exp(-\alpha r^2)r^2r^{d-1}dr\\
& = C\int_{K_1}^{+\infty} \exp(-\alpha r^2)r^{d+1}dr\\
& = C\left(\left[\frac{\exp(-\alpha r^2)r^d}{-2\alpha}\right]_{K_1}^{+\infty}+\frac{d+1}{2\alpha}\int_{K_1}^{+\infty} \exp(-\alpha r^2)r^d dr\right)\\
& = C\left(\left[\frac{\exp(-\alpha r^2)r^d}{-2\alpha}\right]_{K_1}^{+\infty}+\ldots+\frac{(d+1)!}{(2\alpha)^{d+1}}\int_{K_1}^{+\infty} \exp(-\alpha r^2) dr\right),
\end{aligned}  
\end{equation}
where $I_3 \rightarrow 0$ as $K_1 \rightarrow +\infty$. Since $K_1 = \sqrt{-\frac{1}{\alpha} \log\frac{\varepsilon_1}{A_1}}$, we have $I_3 \rightarrow 0$ as $\varepsilon_1 \rightarrow 0$.
\end{proof}

In summary, Theorem \ref{Thm2} shows that the loss function $\mathcal{L}_1$ converges to the entropy of $p_X(\bm x)$ as $\bm h$ converges to $\bm f$.
The following Theorem \ref{Thm3} proves the convergence of the continuous loss function $\mathcal{L}_2$ as defined in Eq.~\eqref{con_l2}.

\begin{theorem}\label{Thm3}
For arbitrarily small $\varepsilon > 0$, there exists a single-hidden-layer neural network $\bm g$ such that 
\begin{equation}\label{thm3_bound}
    \mathcal{L}_2 < \varepsilon,
\end{equation}
in any compact set $\Omega\subseteq\mathbb{R}^d$.
%
\end{theorem}

\begin{proof}
By Assumption \ref{ass1}, $p_X$ is bounded. $\mathcal{L}_2$ has the following bound:
\begin{equation}\label{thm2_l2}
    \begin{aligned}
        \mathcal{L}_2 =\, & \int_{\Omega} p_X(\bm x)\left(\|\bm g(\bm h(\bm x))-\bm x\|^2 d\bm x + \int_{\Omega}|\det {\bf J}_{\bm g}(\bm h(\bm x)) \det {\bf J}_{\bm h}(\bm x)-1|\right)d\bm x\\
        \le \,  &  C\int_{\Omega} \|\bm g(\bm h(\bm x))-\bm x\|^2 d\bm x + \int_{\Omega}|\det {\bf J}_{\bm g}(\bm h(\bm x)) \det {\bf J}_{\bm h}(\bm x)-1|d\bm x\\
        = \, & C\Big(\int_{\Omega} \|\bm g(\bm h(\bm x))-\bm h^{-1} (\bm h (\bm x))\|^2 d\bm x \\
        \, & + \int_{\Omega}|\det {\bf J}_{\bm g}(\bm h(\bm x)) (\det {\bf J}_{\bm h^{-1}}(\bm h(\bm x)))^{-1} - 1|d\bm x\Big)\\
        = \, & C\Big(\int_{\bm h(\Omega)} \|\bm g(\bm y)-\bm h^{-1} (\bm y)\|^2|\det {\bf J}_{\bm h^{-1}}(\bm y)| d\bm y \\
        \, & + \int_{{\bm h}(\Omega)}|\det {\bf J}_{\bm g}(\bm y) - \det {\bf J}_{\bm h^{-1}}(\bm y)|d\bm y\Big):=C(I_1+I_2).
    \end{aligned}
\end{equation}

Since $\bm h$ is continuous in a compact set $\Omega$, by the extreme value theorem, $\bm h$ is bounded in $\Omega$, i.e., there exists a constant $M$ such that $|\bm h(\bm x)| \leq M$ for all $\bm x \in \Omega$. Therefore, $\bm h(\Omega)\subseteq B_d(M)$. Let $\bm g$ be the approximation of $\bm h^{-1}$, and the term $I_1$ has the bound
\begin{equation}\label{thm2_i1}
    \begin{aligned}
        I_1=&\int_{\bm h(\Omega)} \|\bm g(\bm y)-\bm h^{-1} (\bm y)\|^2|\det {\bf J}_{\bm h^{-1}}(\bm y)| d\bm y\\
        \le&\int_{B_d(M)} \|\bm g(\bm y)-\bm h^{-1} (\bm y)\|^2|\det {\bf J}_{\bm h^{-1}}(\bm y)| d\bm y\\
        =&\int_{B_d(M)} \left(\sum_{i=1}^d | g_i(\bm y)- h^{-1}_i (\bm y)|^2 \right)|\det {\bf J}_{\bm h^{-1}}(\bm y)| d\bm y\le\int_{B_d(M)} d\varepsilon_1^2 C d\bm y \le C\varepsilon_1^2,
    \end{aligned}
\end{equation}
where the last inequality holds by Lemma \ref{UAT} with $k = 0$ and the determinate of Jacobian of ${\bm h}^{-1}$ has the upper bound in Eq.~\eqref{thm1_jhlow}.
Let $\mathcal{S}_d$ be the set of all permutation of $\{1,\ldots,d\}$, we have
\begin{equation}\label{thm2_jg-jh-1}
|\det {\bf J}_{\bm g}(\bm y) - \det {\bf J}_{\bm h^{-1}}(\bm y)|
=\left| \sum_{\sigma\in \mathcal{S}_d}\left( {\rm sgn}(\sigma) \left(\prod_{i = 1}^d \frac{\partial g_i(\bm y)}{\partial x_{\sigma_i}} -\prod_{i = 1}^d \frac{\partial h^{-1}_i(\bm y)}{\partial x_{\sigma_i}} \right)\right) \right|< C \varepsilon_1
\end{equation}
for any $\bm y \in B_d(M)$, where the last inequality holds by Lemma \ref{UAT} with $k = 1$. We have the bound for $I_2$
\begin{equation}\label{thm2_i2}
        I_2=\int_{\bm h(\Omega)}|\det {\bf J}_{\bm g}(\bm y) - \det {\bf J}_{\bm h^{-1}}(\bm y)|d\bm y\le\int_{ B_d(M)}|\det {\bf J}_{\bm g}(\bm y) - \det {\bf J}_{\bm h^{-1}}(\bm y)|d\bm y \le C\varepsilon_1.
\end{equation}

The convergence of $\mathcal{L}_2$ to 0 as $\varepsilon_1$ approaches 0 is established by Eqs.~\eqref{thm2_l2}, \eqref{thm2_i1} and \eqref{thm2_i2}. By choosing $\varepsilon_1$ sufficiently small, we complete the proof.
\end{proof}

\subsection{Convergence of the KL divergence}\label{sec:kl}
The analysis in Section \ref{sec:l1l2} shows that the convergence of $\mathcal{L}_1$ and $\mathcal{L}_2$ is a necessary condition 
necessary condition to obtain the convergent approximation discussed in Section \ref{sec:exist}, which verifies that appropriateness of the  loss function used to train the PR-NF model. Here we show that the convergence of $\mathcal{L}_1$ and $\mathcal{L}_2$ plus some mild assumptions on the $\bm h$ and $\bm g$ is sufficient to ensure the convergence of the KL divergence between $p_X$ and $p_{\widehat{X}}$. 


\begin{assum}\label{ass3}
We assume that the target random variable $X$ has finite second-order moment, and there exists positive constant $A_5$ such that
\begin{equation}\label{assum_5}
    \nabla p_{\widehat X}(\bm x)\le A_5(\|\bm x\|+1)p_{\widehat X}(\bm x),
\end{equation}
where $p_{\widehat X}(\bm x)$ is defined in Eq.~\eqref{p_xtilde}.
\end{assum}

\begin{theorem}\label{Thm4}
For any given $\varepsilon>0$, under the assumptions of Theorems \ref{Thm2}, \ref{Thm3} and Assumption \ref{ass3}, the KL divergence between target variable $X$ and the approximation $\widehat{X}$ defined in Eq.~\eqref{eq:xxx} satisfies $D_{\rm KL}(p_X\|p_{\widehat X})<\varepsilon$.
\end{theorem}

\begin{proof}
For arbitrarily small $\varepsilon_1 >0$, under Theorems \ref{Thm2}, we have
\begin{equation}\label{thm4_l1}
\begin{aligned}
    \mathcal{L}_1 \le -\int_{\mathbb{R}^d} p_X(\bm x) \log p_X(\bm x) d\bm x + \varepsilon_1,
\end{aligned}
\end{equation}
and for arbitrarily small $\varepsilon_2 >0$, under Theorems \ref{Thm3}, we have
\begin{equation}\label{thm4_l2}
    \mathcal{L}_2(\mathbb{R}^d) = \int_{\mathbb{R}^d} p_X({\bm x})\left(\|{\bm g(\bm h(\bm x))-\bm x}\|^2 + |\det {\bf J}_{\bm g}(\bm h({\bm x})) \det {\bf J}_{\bm h}({\bm x}) -1|\right)d{\bm x} < \varepsilon_{2}.
\end{equation}
To simplify the notation, we denote $\bm M(\bm x):=\bm g(\bm h(\bm x))$, Then Eq.~\eqref{thm4_l2} is rewritten as 

\begin{equation}\label{thm4_l2m}
    {\mathcal{L}_2(\mathbb{R}^d) = \int_{\mathbb{R}^d} p_X(\bm x)(\|\bm M(\bm x)-\bm x\|^2 + |\det {\bf J}_{\bm M}(\bm x) -1|)d\bm x} < \varepsilon_2.
\end{equation}
Since $\widehat X =\bm g(Z) = \bm g(\bm h(\widetilde X)) = \bm M(\widetilde X)$, by the change of variables formula we have
\begin{equation}\label{thm4_cov}
    p_{\widetilde X}(\bm x) = p_{\widehat X}(\bm M (\bm x))|\det {\bf J}_{\bm M}(\bm x)|.
\end{equation}

The KL divergence between distributions $P_X$ and $P_{\widehat X}$ has the following bound
\begin{equation}
    \begin{aligned}
        D_{\rm KL}(p_X\|p_{\widehat X})\le&\int_{\mathbb{R}^d} p_X(\bm x) \log \frac{p_X(\bm x)}{p_{\widehat{X}}(\bm x)} d\bm x \\
        \le&\int_{\mathbb{R}^d} p_X(\bm x) \log \frac{p_X(\bm x)}{p_{\widetilde X}(\bm x)} dx +\int_{\mathbb{R}^d} p_X(x) \log \frac{p_{\widetilde X}(\bm x)}{p_{\widehat X}(\bm x)} d\bm x \\
        =&\int_{\mathbb{R}^d} p_X(\bm x) \log \frac{p_X(\bm x)}{p_{\widetilde X}(\bm x)} dx +\int_{\mathbb{R}^d} p_X(\bm x) \log \frac{p_{\widehat X}(\bm M (\bm x))|\det {\bf J}_{\bm M}(\bm x)|}{p_{\widehat X}(\bm x)} d\bm x \\
        =&\int_{\mathbb{R}^d} p_X(\bm x) \log \frac{p_X(\bm x)}{p_{\widetilde X}(\bm x)} dx +\int_{\mathbb{R}^d} p_X(\bm x) \log \frac{p_{\widehat X}(\bm M (\bm x))}{p_{\widehat X}(\bm x)} d\bm x \\
        &+ \int_{\mathbb{R}^d} p_X(\bm x) \log |\det {\bf J}_{\bm M}(\bm x)| d\bm x :=I_1+I_2+I_3.
    \end{aligned}
\end{equation}
According to Eq.~\eqref{thm4_l1}, we have $I_1<\varepsilon_1$. Apply the mean value theorem, $I_2$ has the bound
\begin{equation}
    \begin{aligned}
        I_2 =&\int_{\mathbb{R}^d} p_X(\bm x) \log \frac{p_{\widehat X}(\bm M (\bm x))}{p_{\widehat X}(\bm x)} d\bm x \\
        =&\int_{\mathbb{R}^d} p_X(\bm x) \Big(\log (p_{\widehat X}(\bm M (\bm x)) )-\log (p_{\widehat X}(\bm x))\Big) d\bm x\\
        =&\int_{\mathbb{R}^d} p_X(\bm x)\left(\frac{\nabla p_{\widehat X}(\bm \xi)}{p_{\widehat X}(\bm \xi)} \left\| \bm M (\bm x)-\bm x\right\| \right)d\bm x,
    \end{aligned}
\end{equation}
where $\bm \xi$ is between $\bm M(\bm x)$ and $\bm x$. By Eq.~\eqref{assum_5} in Assumption \ref{ass1}, we have
\begin{equation}
    \begin{aligned}
        I_2 \le \, &C\int_{\mathbb{R}^d} p_X(\bm x)(\|\bm \xi\|+1)\left\| \bm M (\bm x)-\bm x\right\|d\bm x\\
        \le \, &C\int_{\mathbb{R}^d} p_X(\bm x)\left(\|\bm M(\bm x)-\bm x\| + \|\bm x\| +1\right)\left\| \bm M (\bm x)-\bm x\right\|d\bm x\\
        = \, &C\bigg(\int_{\mathbb{R}^d} p_X(\bm x)\left\| \bm M (\bm x)-\bm x\right\|^2 d\bm x + \int_{\mathbb{R}^d} p_X(\bm x) \|\bm x\|\left\| \bm M (\bm x)-\bm x\right\|d\bm x\\
        \, &+\int_{\mathbb{R}^d} p_X(\bm x) \left\| \bm M (\bm x)-\bm x\right\|d\bm x\bigg)\\
        \le \, &C\bigg[\varepsilon_2 + \left(\int_{\mathbb{R}^d} p_X(\bm x)\|\bm x\|^2 d\bm x\right)^{\frac{1}{2}} \left(\int_{\mathbb{R}^d} p_X(\bm x)\left\| \bm M (\bm x)-\bm x\right\|^2 d\bm x\right)^{\frac{1}{2}}\\
        \, &+\left(\int_{\mathbb{R}^d} p_X(\bm x) d\bm x\right)^{\frac{1}{2}} \left(\int_{\mathbb{R}^d} p_X(\bm x)\left\| \bm M (\bm x)-\bm x\right\|^2 d\bm x\right)^{\frac{1}{2}}\bigg]\le C(\varepsilon_2+\varepsilon_2^\frac{1}{2}).
    \end{aligned}
\end{equation}
Since $\log |t|\le|t-1|$ for every $t\in\mathbb{R}$, $I_3$ has the bound
\begin{equation}
        I_3=\int_{\mathbb{R}^d} p_X(\bm x) \log |\det {\bf J}_{\bm M}(\bm x)| d\bm x
        \le\int_{\mathbb{R}^d} p_X(\bm x) |\det {\bf J}_{\bm M}(\bm x)-1| d\bm x\le\varepsilon_2.
\end{equation}
Therefore,
\begin{equation}
    D_{\rm KL}(p_X\|p_{\widehat X})\le\varepsilon_1+C(\varepsilon_2+\varepsilon_2^\frac{1}{2}).
\end{equation}
Let $\varepsilon_1<\varepsilon/2$ and $\varepsilon_2<\min\{\varepsilon^2/(16C^2), 1\}$, then we have
\begin{equation}
    \begin{split}
        D_{\rm KL}(p_X\|p_{\widehat X})\le\varepsilon_1+C(2\varepsilon_2^\frac{1}{2})
        \le\frac{\varepsilon}{2}+C\frac{\varepsilon}{2C}
        \le\varepsilon.
    \end{split}
\end{equation}
Thus, $D_{\rm KL}(p_X\|p_{\widehat X}) \le \varepsilon$. We complete the proof.   
\end{proof}

\section{Numerical examples and applications}\label{sec:ex}
In this section we present numerical experiments demonstrating the superior performance of the PR-NF model in comparison with the standard MC method. The example in Sec.\ref{sec:ex1} benchmarks the accuracy of our method for a problem with  known analytical ground-truth solution. 
In Sec.~\ref{sec:ex2} we present an application to plasma physics and in Sec.\ref{sec:exfluid} an application to fluid mechanics. 
Since in these cases there are not known analytical solutions, 
the ground-truth corresponds to MC simulations with
sufficiently large samples. 
In all numerical simulations, the PyTorch machine learning framework has been implement with CUDA GPU. The PR-NF model uses the Adam optimizer \cite{kingma2014adam} and the Tanh activation function.

\subsection{Verification of algorithm accuracy}\label{sec:ex1}
In this subsection, we use a one-dimensional nonlinear SDE and a ten-dimensional linear SDE to verify the accuracy of the proposed algorithm. We first consider the following one-dimensional SDE
\begin{equation}\label{eq_ex1}
X_t  =  X_0 + \int_0^t {\bm b}(s, X_s) ds
+  \int_0^t {\bm \sigma}(s,  X_s) d  W_s \;\; \text{ with } X_0 \in [0,L],
\end{equation}
where ${\bm b}(s, X_s) = 2\sqrt{X_s} + 1$, ${\bm \sigma}(s, X_s) = 2\sqrt{X_s}$, $L=5$, $t= 0.1$. The analytical solution of this SDE is
\begin{equation}\label{eq_ex1solu}
    X_t = (\sqrt{X_0} + t + W_t)^2,
\end{equation}
where $X_0$ is the initial condition, and the corresponding conditional distribution is 
\begin{equation}
\label{eq_ex1_pdf}
 p_{X_t|X_0}( x|x_0) = \frac{1}{{2\sqrt{2\pi tx}}}\left[
 \exp{\left(\frac{-(\sqrt{x}-\sqrt{x_0}-t)^2}{2t}\right)}
 + \exp{\left(\frac{-(\sqrt{x}+\sqrt{x_0}+t)^2}{2t}\right)}
 \right].
\end{equation}
The distribution of $X_t$ subject to an initial distribution $p_{X_0}(x_0)$ can be computed by the convolution with $p_{X_t|X_0}( x|x_0)$. Since the analytical solution is known, we use this example to test the accuracy of the propose method.

The training dataset consists of $N_{\rm train} = 20000$ samples with initial positions $\{X_0^{(n)}\}_{n = 1}^{N_{\rm train}}$ sampled from a uniform  distribution over the domain $\mathcal{D} = [0,5]$, and terminal positions $\{X_t^{(n)}\}_{n = 1}^{N_{\rm train}}$ generated by MC simulations of Eq.~\eqref{eq_ex1}.
The neural network has $N_{\rm layer} = 1$ hidden layer with $N_{\rm neuron} = 256$ neurons. 
The hyperparameter $\lambda$ in the loss function is selected according to the method described in  Sec.~\ref{sec:hyper}. The left panel of Fig.~\ref{ex1_crossentr} shows the cross entropy $H(\lambda)$ defined in Eq.~\eqref{eq:cross} with different $\lambda$. As it shown, the PR-NF model attains the minimum of $H(\lambda$) when $\lambda = 50$. The middle panel and right panel of Fig.~\ref{ex1_crossentr} correspond to the accuracy performance of the proposed PR-NF model with two cases $\lambda = 1$ and $\lambda = 50$, respectively. The red curve is the exact density function and the histogram is synthesized by the PR-NF model. A good agreement is shown in the right panel ($\lambda = 50$). In contrast, $\lambda = 1$ is too small to guarantee the reversibility. Without the tuning process of $\lambda$, the PR-NF model may fail to balance the $\mathcal{L}_1$ and $\mathcal{L}_2$, e.g., the middle panel of Fig.~\ref{ex1_crossentr}. The minimization of cross entropy $H(\lambda)$ in Section \ref{sec:hyper} is an effective and appropriate approach to choose $\lambda$. 
An optimal parameter $\lambda$ is necessary to obtain an accurate surrogate model. 
\begin{figure}[h!]
\settoheight{\tempdima}
\centering\begin{tabular}{@{}c@{ }c@{ }c@{}}
\vspace{-0.2cm}
\includegraphics[width=.32\linewidth]{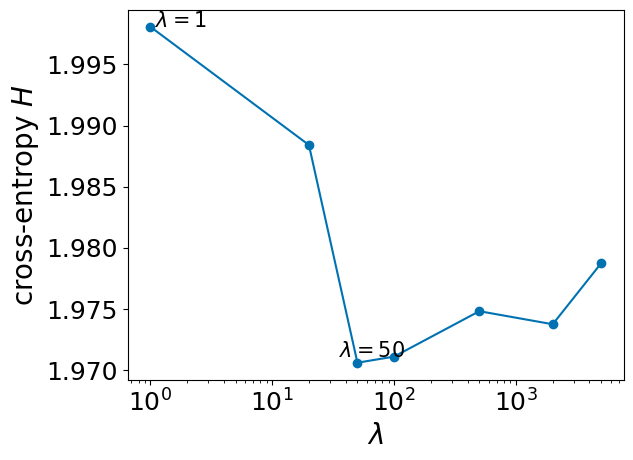}&\hspace{-0.1cm}
\includegraphics[width=.32\linewidth]{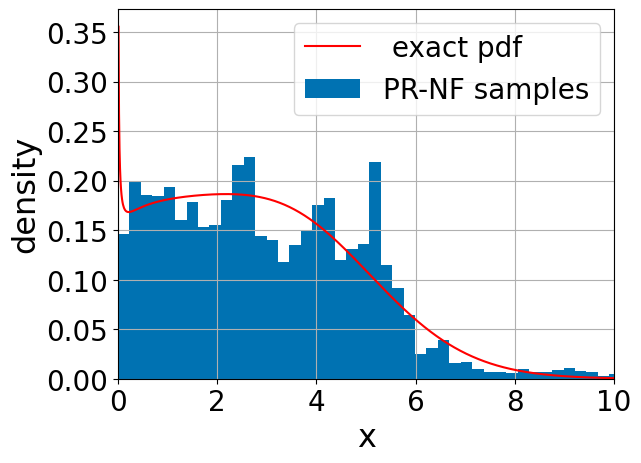}&\hspace{-0.25cm}
\includegraphics[width=.32\linewidth]{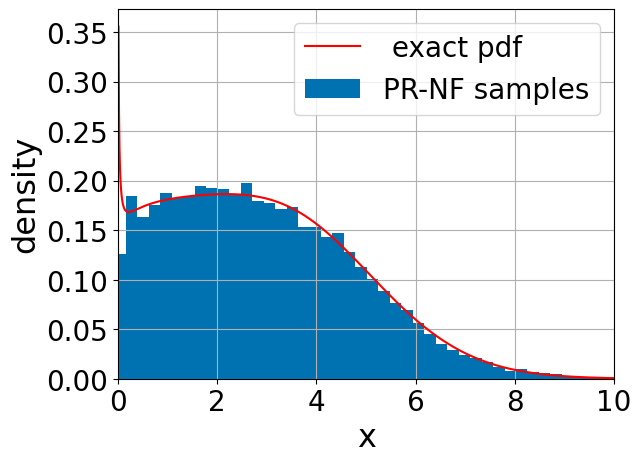}\\[-1.5ex]
\end{tabular}
\caption{The left panel shows the cross entropy $H(\lambda)$, which attains the minimun when $\lambda = 50$. The middle and right panels show the fitting performance for $\lambda = 1$ and $\lambda = 50$, respectively. It is observed that the accuracy of PR-NF model with $\lambda = 50$ is better than the case $\lambda = 1$. Without the tuning process of $\lambda$, the PR-NF model may fail to balance the $\mathcal{L}_1$ and $\mathcal{L}_2$ (e.g., $\lambda = 1$ case). The minimization of the cross entropy $H(\lambda)$ is an effective and correct approach to choose $\lambda$. 
The PR-NF model with an optimal parameter $\lambda$ ($\lambda = 50$ case) achieves a good fitting performance.}%
\label{ex1_crossentr}
\end{figure}

We now use the well-trained PR-NF model ($\lambda = 50$) to sample $X_t$ with variable initial distributions. The top row of Fig.~\ref{fig_ex1} shows four different initial distributions, denoted as ``$\delta$ func", ``bar", ``sin2", and ``ricker". The bottom row of Fig.~\ref{fig_ex1} shows corresponding histogram of $X_t$ generated by the PR-NF model, where a good agreement with the exact density is observed.  The results demonstrate that the well-trained PR-NF model can handle different initial conditions without additional training process.  Figure~\ref{Ex1_1DKL} shows the decay of the loss function of training dataset (left panel) and the decay of the Kullback–Leibler (KL) divergence of the four initial distributions (right panel). The KL divergence formula is defined as the relative entropy from the approximate density (generated from PR-NF) $p_{\rm approx}$ to the exact density $p_{\rm exact}$, i.e.,
\begin{equation}\label{eq_kl}
  D_{\rm KL}(p_{\rm exact} \,\| \,p_{\rm approx}) = \int_{-\infty}^{\infty} p_{\rm exact}(x)\log{\left(\frac{p_{\rm exact}(x)}{p_{\rm approx}(x)}\right)}dx,
\end{equation}
where $D_{\rm KL}$ can be approximated by the Riemann sum over a uniform mesh or the mean value of $\log{\left(\frac{p_{\rm exact}(x)}{p_{\rm approx}(x)}\right)}$ over 
the training dataset $\mathcal{X}$ based on the exact solution in Eq.~\eqref{eq_ex1solu}. We monitor the KL divergence for different distributions during the training process. 
The consistency of decays between loss function and KL divergence  illustrates that the loss function $\mathcal{L}$ defined in Eq.~\eqref{eq_loss} involving the reversibility error is effective to be regarded as a loss function  for normalzizing flow. 

\begin{figure}[h!]
\begin{center}
\includegraphics[scale =0.2]{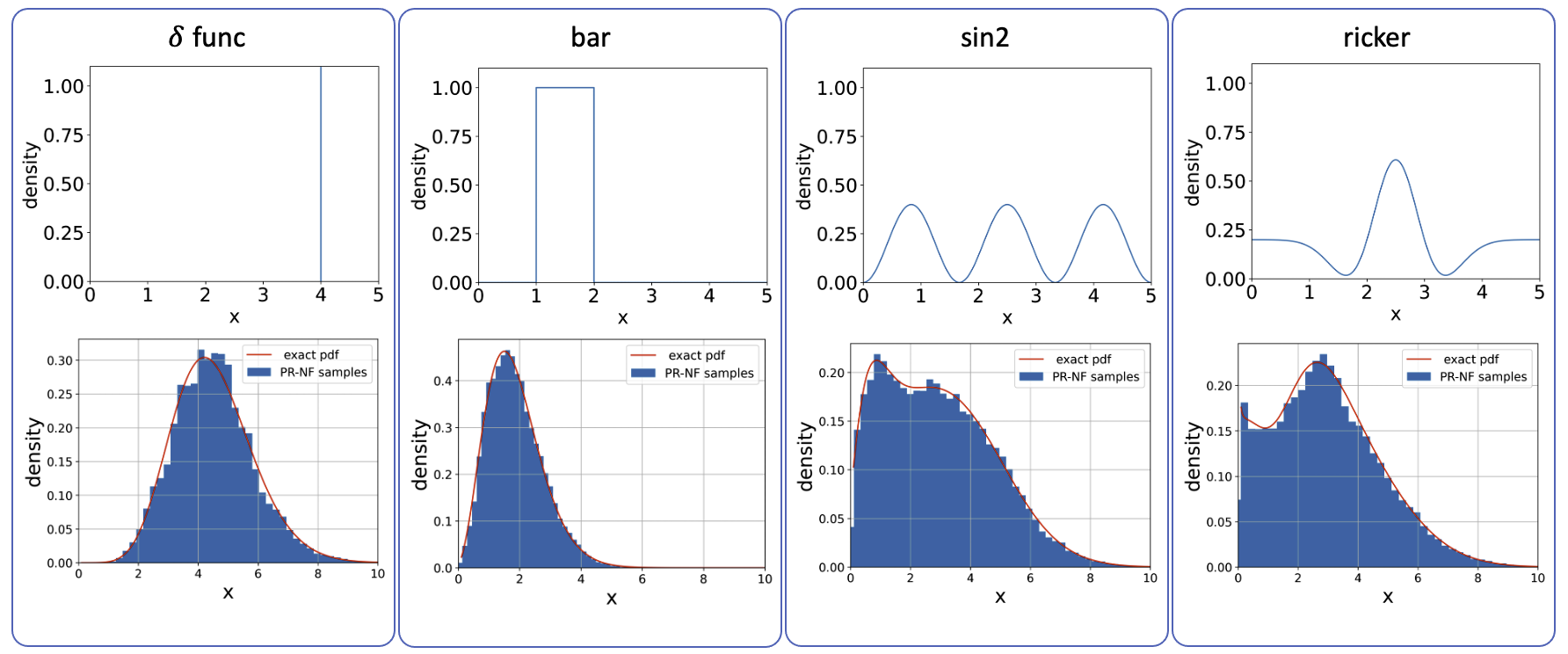}
\caption{The accuracy performance of the well-trained PR-NF model for test dataset. The top row shows the initial distributions of four dataset: $\delta$ func, bar, sin2, and ricker. The corresponding plots on the bottom row present the fitting between the exact density of $X_t$ (red curve) and the histogram (generated by PR-NF model) at $t = 0.1$. The results demonstrate that the well-trained PR-NF model can handle various initial conditions without re-training.}%
\label{fig_ex1}
\end{center}
\end{figure}
\begin{figure}[h!]
\begin{center}
\includegraphics[scale =0.4]{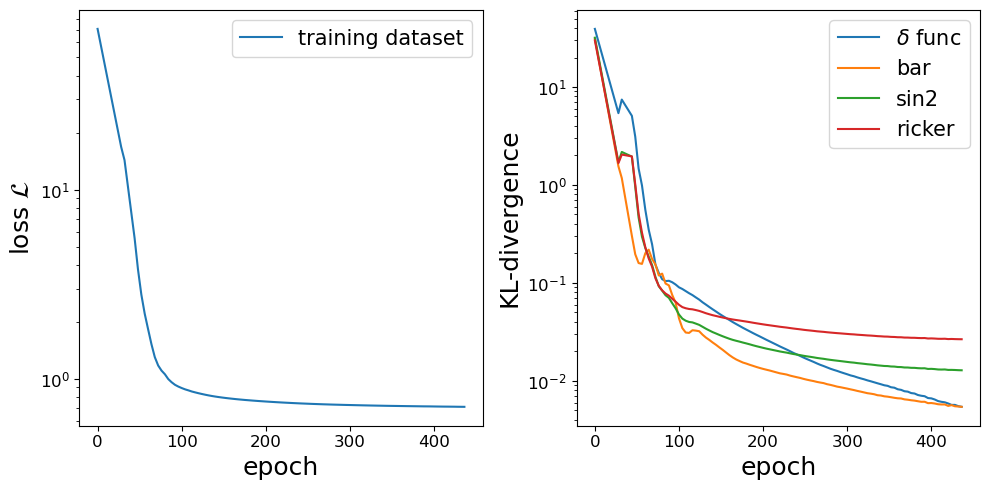}
\vspace{-0.2cm}
\caption{The decay of loss function of training dataset (left panel) and the decay of KL divergence of four test dataset (right panel). We monitor the KL divergence during the training process. The consistency  of decays between loss function and KL divergence demonstrates that the loss function $\mathcal{L}$ defined in Eq.~\eqref{eq_loss} involving the reversibility error is effective to be regarded as a loss function for normalzizing flow.}
\label{Ex1_1DKL}
\end{center}
\end{figure}




To illustrate the application of our method to high-dimensional problems, consider the ten-dimensional SDE
\begin{equation}\label{eq_10d}
X_t  =  X_0 + \int_0^t X_s ds
+  \int_0^t K X_s d  W_s \;\; \text{ with } X_0 \in \mathcal{D},
\end{equation}
where  $t=1$, $\mathcal{D}= [0,1]^{10}$,
{\tiny $K = \frac{1}{2}
    \begingroup 
\setlength\arraycolsep{10pt}
\begin{pmatrix}
    1 & 1 & & \\
 & \ddots & \ddots & \\
 &  & \ddots & 1 \\
 & & & 1
\end{pmatrix}_{10 \times 10}.
\endgroup
$}
In this case the analytical solution $X_t$ is given by
\begin{equation}\label{eq_10dsolu}
    X_t = \exp((\mathbb{I}_{10}- K ^2/2)t + KW_t))X_0.
\end{equation}
The training dataset consists of $N_{\rm train} = 20000$ samples where $X_0$ are sampled from an uniform initial distribution over the domain $\mathcal{D}$. The neural network has $N_{\rm hidden} = 1$ hidden layer with $N_{\rm neuron} = 256$ neurons. The left panel of Fig.~\ref{ex1_10d} shows the cross entropy as function of $\lambda$, with a  minimum at $\lambda = 100$. 

We use the well-trained PR-NF model to test four initial distributions, one is the normal distribution $x_{\rm test} \sim \mathcal{N}(0.5, 0.1\cdot\mathbb{I}_{10})$, other three are nonlinear transformations of the normal distribution, square:= $x_{\rm test}^2$, log:= $\ln (|x_{\rm test}| + 1)$, and sin:= $\sin{(x_{\rm test}^2)}$. The middle panel of Fig.~\ref{ex1_10d} shows the decay of loss function of the training dataset. And the decay of KL divergence of four test dataset is shown in the right panel. A good consistency is observed.

\begin{figure}[h!]
\settoheight{\tempdima}
\centering\begin{tabular}{@{}c@{ }c@{ }}
\includegraphics[width=.33\linewidth]{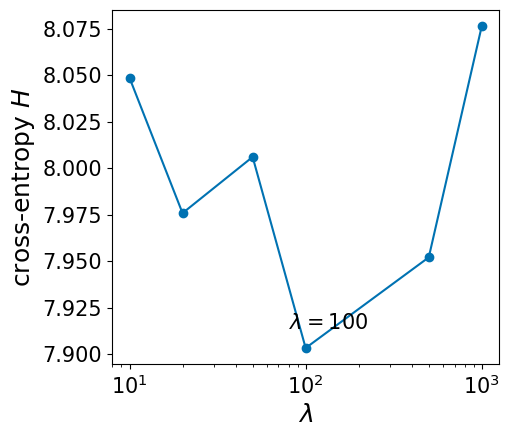}&\hspace{-0.31cm}
\includegraphics[width=.66\linewidth]{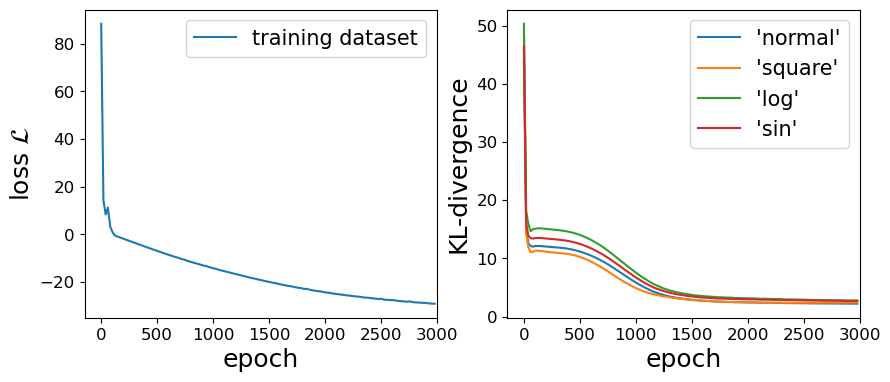}\\[-1.5ex]
\end{tabular}
\caption{Numerical results for the ten-dimensional Brownian motion case. The left panel shows the cross entropy $H(\lambda)$ with different $\lambda$, where it attains the minimum when $\lambda = 100$. Moreover, we evaluated $H(\lambda)$ for $\lambda = 1, 2000$ (not shown). It is large when $\lambda = 2000$ and is relative small when $\lambda = 1$. However, the loss decay is noisy for $\lambda = 1$.  The middle panel shows the decay of loss function and the right panel shows the decay of KL divergence of four test dataset using the well-trained PR-NF model. Since we are facing a ten-dimensional problem, the density values for samples are very small. That leads to a relatively large KL divergence that happens to the high-dimensional density estimation \cite{rezende2015variational}. We also tested the marginal KL divergence, which is on the order of $10^{-2}$.}%
\label{ex1_10d}
\end{figure}


\subsection{Generation of runaway electrons in magnetically confined  nuclear fusion plasmas}\label{sec:ex2}

Modeling and simulation of plasmas is an extremely complex computational problem involving, among other aspects, the trajectories of the plasma particles (electrons and ions), the evolution of external and self-generated electromagnetic fields, as well as energy and particle sources and sinks. 
Particle-based simulations (usually known as Monte Carlo simulations in the plasma physics literature) are based on the solution of stochastic differential equations of the form in Eq.(1)
with $X_t$ denoting the particles' coordinates in the $d$-dimensional phase space. In the full-orbit description $d=6$, but reduced models with $d<6$ are commonly used.  The deterministic term, ${\bm b}\, ds$, is given by the electromagnetic forces, and the stochastic term, ${\bm \sigma\, dW_s}$, is used to model collisional and diffusive processes.  This is a strongly multi-scale problem since there is a big disparity of spatiotemporal scales in the evolution of  the orbits and the electromagnetic fields. In addition, to avoid statistical sampling errors, these equations need to be solved for very large number of particles and large sets of different initial conditions. The goal of this section is to discuss how the proposed PR-NF method can be used to overcome some of these computational challenges in the case  of magnetically confined high temperature plasmas in toroidal  geometry. This configuration (known as tokamak) is the leading approach to achieve the elusive goal of controlled nuclear fusion for electricity production. The specific problem we consider is the simulation of relativistic runaway (RE) electrons produced during magnetic disruptions, see for example Ref.~\cite{breizman2019physics} and references therein. Understanding this problem is critical for the safe operation of nuclear fusion reactors.  

The model of interest consists of the following set of SDEs
\begin{equation}\label{ex32d}
\left\{
\begin{aligned}
dp &=  \left [E \xi\, - \frac{\gamma p}{\tau}(1-\xi^2)-C_F(p) +\frac{1}{p^2}\frac{ \partial }{\partial p} \left(p^2 C_A\right) \right ] dt + \sqrt{2 C_A} \, dW_p,\\ 
d \xi &= \left[ \frac{E \left(1-\xi^2\right)}{p} + \frac{\xi (1-\xi^2)}{\tau \gamma}
-2 \xi \frac{C_B}{p^2} \right ] dt + \frac{\sqrt{2 C_B}}{p}  \,  \sqrt{1-\xi^2}\, dW_\xi \, .
\end{aligned}
\right.
\end{equation}
This two-dimensional reduced model
tracks the evolution of the magnitude of the relativistic momentum, $p$, and 
$\xi=\cos \eta$, where $\eta$ is the pitch angle between the particle velocity and the magnetic field which is assumed fixed and given.  
 $dW_p$ and $dW_\xi$ are independent standard Brownian motions, $E$ is the electric field, and $\tau$ the radiation damping time scale. The collision operator includes the momentum diffusion, $C_A$, pitch angle scattering, $C_B$, and Coulomb drag, $C_F$, which are defined as 
$$
C_A (p) = \bar{\nu}_{ee} \, \bar{v}_T^2 \,\,\frac{\psi(y)}{y},  \quad C_F (p)=2\,\bar{\nu}_{ee}  \, \bar{v}_T \, \psi(y),
$$
$$
C_B (p) = \frac{1}{2} \,\bar{\nu}_{ee} 
\, \bar{v}_T^2 \, \, \frac{1}{y}  \left[ Z + \phi(y)- \psi(y) +  \frac{y^2}{2} \delta^4 \right], \nonumber
$$
$$
\phi(y)=\frac{2}{\sqrt{\pi}} \int_0^y e^{-s^2} ds, \quad
\psi(y)=\frac{1}{2 y^2} \left[ \phi(y)-y \frac{d \phi}{dy} \right], \quad y = \frac{1}{\bar{v}_T} \frac{p}{\gamma},
$$
$$
\gamma = \sqrt{1 + (\tilde{\delta} p)^2}, \quad \tilde{\delta} = \frac{\tilde{v}_T}{c} = \sqrt{\frac{2 \tilde{T}}{mc^2}}, \quad {\delta} = \frac{\hat{v}_T}{c} = \sqrt{\frac{2 \hat{T}_f}{mc^2}}, \quad
\bar{v}_T = \sqrt{\frac{\hat{T}_f}{\tilde{T}}},
$$
$$
\bar{\nu}_{ee} = \left( \frac{\tilde{T}}{\hat{T}_f} \right)^{3/2} \frac{\rm{ln \hat{\Lambda}}}{\rm{ln \tilde{\Lambda}}}, \quad {\rm ln} \tilde{\Lambda} = 14.9 - \frac{1}{2} {\rm ln}0.28 + {\rm ln} \Tilde{T}, \quad {\rm ln} \hat{\Lambda} = 14.9 - \frac{1}{2} {\rm ln}0.28 + {\rm ln} \hat{T}_f,
$$
with $Z$ and $c$ denoting the ion effective charge and the speed of light, respectively. 
Further details on this model can be found in Ref.~\cite{del2021generation, yang2021feynman} and references therein. 

We are interested in the hot-tail acceleration of electrons due to the rapid cooling of the plasma (thermal quench) during a magnetic disruption  
\cite{stahl2016kinetic,petrov2021numerical}.
We use a fast cooling model for the temperature and an electric field from Ohm's law with a Spitzer resistivity temperature dependent model
$ E = E_0 \left[ \frac{\tilde{T}}{\hat{T}_f} \right]^{3/2}$, 
with typical model parameters
 $Z = 1$, $\tau = 6 \times 10^3$, $E_0 = 1/2000$, $\tilde{T} = 3$, 
 $\hat{T}_f = 0.05$ and $ mc^2 = 500$.
The integration domain is $p\in (p_{\min}, p_{\max})$ and $\theta \in (0, \pi)$, where 
$(p_{\min}, p_{\max})=(0.5, 5)$, and $t_{\max}=26$. 
The different cases of initial conditions are obtained by  sampling a family of Maxwellian distributions 
\begin{equation}\label{eq_max}
    f(p,\xi,t_0) = \frac{2p^2}{\pi^{1/2} p_0^3} e^{ -(p / p_0)^2 }, \quad \text{where} \,\,p_0 = \sqrt{\frac{\hat{T}_0}{\tilde{T}}},
\end{equation}
normalized as $\int_{0}^1\int_{-1}^{1}\int_{p_{\min}}^{p_{\max}} f_0 (p,\xi) dp d\xi =1$, and parameterized by $\hat T_0$. 
Exploration of the hot-tail generation of RE in this case requires the solution of this problem for different values of $\hat{T}_0$. In the standard Monte Carlo approach this is a computational intense problem because each time $\hat{T}_0$ changes the problem needs to be recomputed for the new initial condition. However, as it will be shown bellow,  this is not the case with the  proposed PR-NF method.

In this problem the training dataset $\mathcal{V}$ consists of $N_{\rm train} = 30000$ samples with initial positions, $\{(p_0,\xi_0)^{(n)}\}$, $n = 1, \ldots,N_{\rm train}$,  uniformly distributed in the integration domain and terminal positions, $(p_{t_{\max}},\xi_{t_{\max}})^{(n)}$,  obtained by a direct  Monte Carlo simulation of Eqs.~\eqref{ex32d}. The neural network architecture of the PR-NF model uses $N_{\rm hidden} = 1$ hidden layer and $N_{\rm neuron} = 512$ neurons. 
To choose the hyperparameter $\lambda$, which  controls the relative importance of the pseudo-reversivility and the negative log-likelihood losses,
we follow the same procedure as in Section~\ref{sec:ex1}. Namely, we select the value of $\lambda$ for which the cross entropy, $H(\lambda)$, has a minimum, which in this case corresponds to $\lambda=100$.

To check the accuracy of the PR-NF model we consider a 
test set of  $N_{\rm test} = 20000$ samples drawn from the Maxwellian distribution in Eq.~\eqref{eq_max} with $\hat{T}_0 = 10$. 
Fig.~\ref{ExRE_pdf} compares the final probability density function (pdf) at $t_{\rm max}=26$ computed with the PR-NF method and the direct Monte Carlo simulation. To construct  the pdf from the $N_{\rm test}$ data points, we use a Gaussian kernel density estimation. The plots on the top row show the one-dimensional marginal pdfs for the pitch angle $\theta$ (left) and momentum $p$ (right).  The plots on the bottom row show the two-dimensional contour plots of the $\log_{10}$ of the pdfs for the PR-NF (left) and the MC (right) simulations. The local pointwise differences between PR-NF and MC methods are of the order $10^{-2}$, 
which implies a small error in the evaluation of quantities of interest like the  production of runaway electrons shown in Fig.~\ref{ExRE_RE}.

\begin{figure}[h!]
\begin{center}
\includegraphics[scale =0.5]{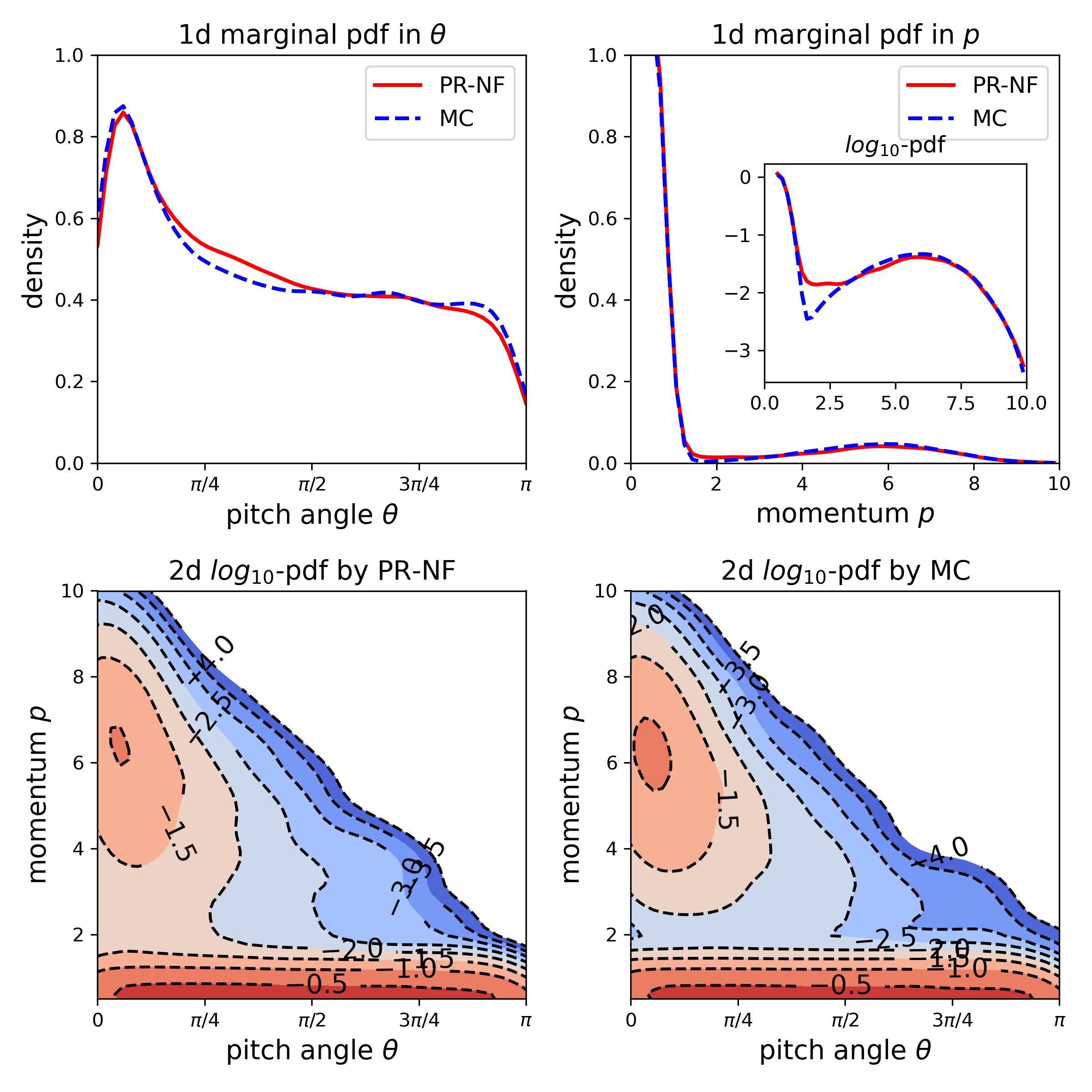}
\vspace{-0.2cm}
\caption{Final state of runaway electrons probability density function, $f(\theta,p,t_{max})$, according to the model in Eq.~~(\ref{ex32d}) with initial condition, $f(\theta,p,t_0)$ in Eq.~(\ref{eq_max}) with $\hat{T}_0 = 10$. The top row compares the $1d$ marginal distributions for the pitch angle (left) and the momentum (right) computed using the PR-NF and the MC methods. The inset in the the top plot on the right highlights, in logarithmic scale, the  bump in the distribution resulting from the acceleration of the hot tail. The contour plots at the bottom compare the $2d$ distributions using a $\log_{10}$ scale. The local pointwise differences between PR-NF and MC methods are of the order $10^{-2}$. Such agreement between MC and the PRNF methods is sufficient to ensure the accuracy of quantity of interest, for example, the runaway electron production rate in Fig.~\ref{ExRE_RE}.
}
\label{ExRE_pdf}
\end{center}
\end{figure}

To illustrate the performance of the PR-NF method in the computation of the final state for different initial conditions, we consider the quantity of interest
\begin{equation}\label{eq_pr}
n_{\rm RE} = \int_{-1}^{1}\int_{p^*}^{p_{\max}} f_{t_{\max}} (p,\xi) dp d\xi,
\end{equation}
representing the total fraction of runaway electrons produced by the hot-tail mechanism 
during the thermal quench, where
$f_{t_{\max}}$ is the density function of electrons at $t_{\max} = 26$,   $p^*=1.75$ is the RE energy threshold and $p_{\max}$ is a numerical upper bound. 
We consider a family of Maxwellian initial conditions of the form in Eq.~\eqref{eq_max}
with $\hat{T}_0 = 1, 2, 3, 4, 5, 6, 7, 8, 9, 10.$
Figure~\ref{ExRE_RE} compares $n_{\rm RE}$ 
computed using a direct MC simulation that requires to solve the whole problem from the start for each value of $\hat T_0$, with the PR-NF method that reduces the computation to the direct evaluation of the final state using the trained neural network. 
This numerical simulation  is based on $N_{\rm sample}=20000$ samples for different $\hat{T}_0$. By using a single PR-NF model, we can obtain the orbits of particles for ten initial conditions without any additional simulations or training. However, the Monte Carlo method has to repeatedly simulate the particles' orbits for ten times. For the simulation, the time step size of the MC method is $\Delta t = 0.001$ and the number of epochs of the PR-NF method is $N_{\rm epoch} = 20000$. We record the running wall-clock time (measured using the python function ``time.time()") between these two methods. For the results in Fig.~\ref{ExRE_RE}, the error between the PR-NF and MC method (viewed as the ground-truth) is acceptable. The accuracy of PR-NF model is enough to predict and analyze the profile of fraction runaway electrons production rate $n_{\rm RE}$. Moreover, the total running time of the MC method of these ten cases of Maxwellian distribution is around $C_{\rm MC} = 4000$ sec. The total running time of the PR-NF method consists of two parts, the offline cost is around $C_{\rm offline} =1354$ sec and the online cost is around $C_{\rm online} = 12$ sec. The PR-NF model is faster that the MC method, especially on the online cost. This efficiency advantage is more valuable when we handle a large number of Maxwellian distributions.

\begin{figure}[h!]
\begin{center}
\includegraphics[scale =0.6]{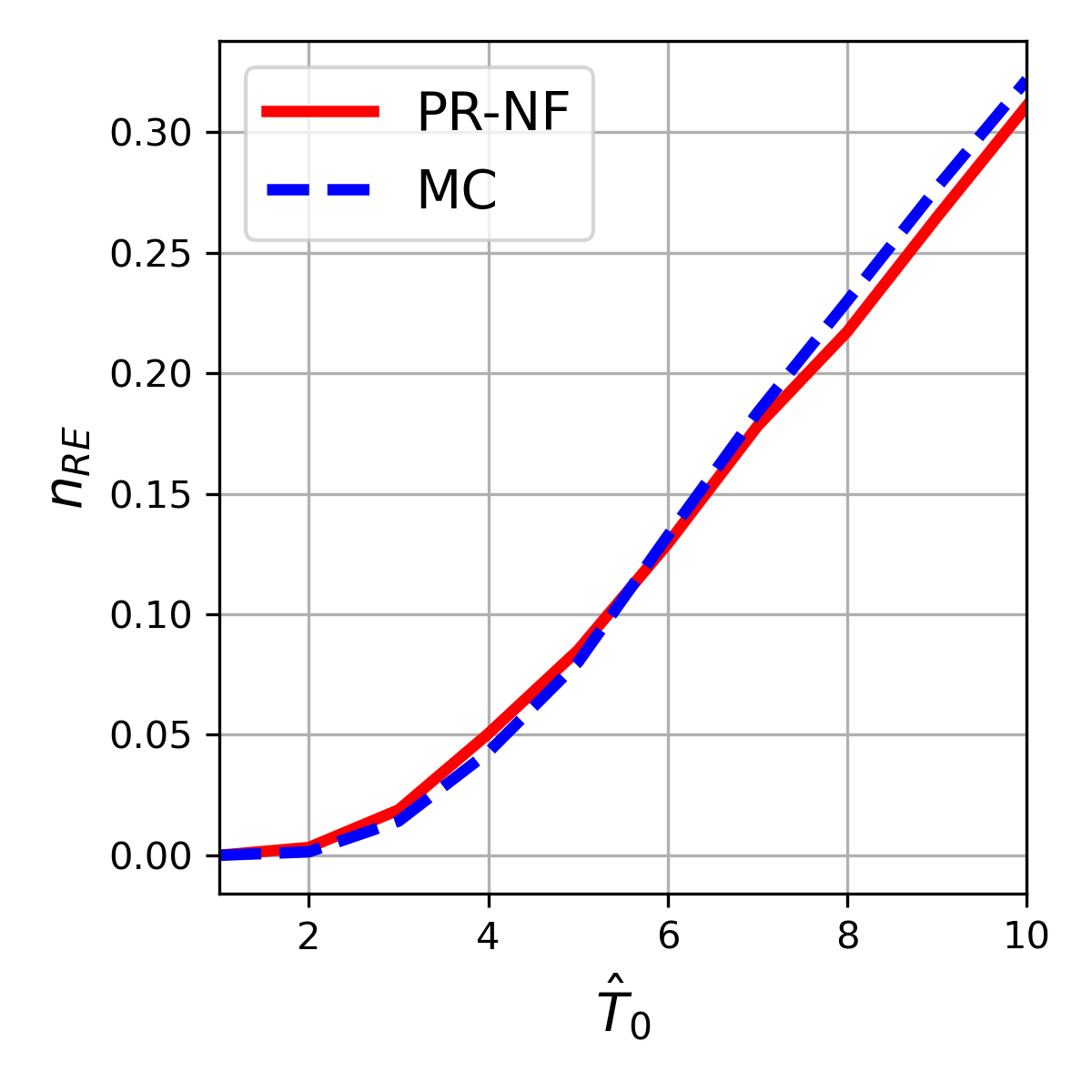}
\vspace{-0.2cm}
\caption{
Quantity of interest $n_{\rm RE}$ in Eq.~\eqref{eq_pr} represents the  production of runaway electrons  as function of the parameter $\hat{T}_0$ in Eq.~(\ref{eq_max}). The red curve shows the  result using the PR-NF surrogate model and the blue curve the  result using the direct MC method. A single PR-NF model can compute $n_{\rm RE}$ for ten initial conditions without
any additional simulations or training, where the online cost is much cheaper than MC method. The accuracy of PR-NF model
is sufficient to predict the profile of the production rate $n_{\rm RE}$ as the MC method (around 1\% absolute error is acceptable here).}
\label{ExRE_RE}
\end{center}
\end{figure}

\subsection{Passive scalar advection-diffusion  transport in a 3D chaotic flow}\label{sec:exfluid}

Understanding passive scalar transport is a problem of significant interest in fluid dynamics in general and environmental engineering, oceanography, and atmospheric sciences in particular. By passive we mean that the transported scalar exactly follows the given velocity field without modifying it. 

As a specific example to illustrate the application of the PR-NF method  to this problem, we consider the ABC (Arnold-Beltrami-Childress) flow \cite{dombre1986chaotic} which 
is a three-dimensional incompressible velocity field which is an exact solution of Euler's equation. 
In Cartesian coordinates this velocity field has components ${\bf v}=(v_x,v_y,v_z)$ with $v_x=A \sin{z} + C \cos{y}$, $v_y= B \sin{x} + A \cos{z}$ and $v_z= C \sin{y} + B \cos{x}$. For a wide range of the $A$, $B$ and $C$ parameters, this flow is not integrable.
That is, it exhibits chaotic advection and the sensitive dependence on initial condition gives rise to rapid mixing and long-range transport \cite{shlesinger1993strange}. 

In addition to advection, passive scalar transport is affected by diffusion. In particle based simulations this important effect can be modeled by adding a stochastic term. Following this approach we consider the following set of SDEs 
\begin{equation}\label{eq_exfluid}
\left\{
\begin{aligned}
dx &=  P_e\, [A \sin{z} + C \cos{y}] \, dt +\, dW_x \, ,\\ 
dy &=  P_e\, [B \sin{x} + A \cos{z}] \, dt +\, dW_y \, ,\\ 
dz &=  P_e\, [C \sin{y} + B \cos{x}] \, dt +\, dW_z \, . 
\end{aligned}
\right.
\end{equation}
where $dW_x$, $dW_y$ and $dW_z$ are independent Brownian motions, and the Pecl\'e number, $P_e$ is a dimensionless parameter controlling the relative importance of advection and diffusion. When, $P_e \ll  1$, transport is dominated by diffusion and when, $P_e \gg  1$, transport is dominated by advection.

The initial conditions,  $\{x_0^{(n)}, y_0^{(n)}, z_0^{(n)}\}_{n=1}^{N_{\rm train}}$, of the training dataset consists of an ensemble of $N_{\rm train} = 30000$ random points uniformly distributed  in the cubic domain  $\mathcal{D} =  [0,2\pi] \times [0,2\pi] \times [0,2\pi]$. 
The terminal positions, $\{x_{t_{\max}}^{(n)}, y_{t_{\max}}^{(n)}, z_{t_{\max}}^{(n)}\}_{n=1}^{N_{\rm train}}$, 
of the training dataset are obtained by a direct Monte Carlo simulation of 
 Eq.~\eqref{eq_exfluid}. 
 In the calculation presented we use  $t_{\rm max} = 2$, and model parameters $P_e = 3$,  $A = 1$, $B = 1$, and  $C = 0.25$. 
 For the PR-NF neural network model we use  $N_{\rm hidden} = 1$ hidden layer with $N_{\rm neuron} = 512$ neurons.
Once the training process is done, the PR-NF model can be used as a surrogate model to predict the final state for various initial conditions, without the need to integrate Eqs.~\eqref{eq_exfluid}.

As a first test, we apply the PR-NF to predict the final state of an initial condition of the form
\begin{equation}\label{exfluid_initial}
    f(x,y,z,t_0) = H(x,y,z) \exp{\left[ - \left(\frac{x-x_c}{\sigma_x}\right)^2 - \left(\frac{y-y_c}{\sigma_y}\right)^2 - \left(\frac{z-z_c}{\sigma_z}\right)^2\right]},
\end{equation}
modeling, for example, a Gaussian-shaped cloud of a pollutant, with $\sigma_x = \pi/3$, $\sigma_x = \pi/5$, $\sigma_x = \pi/4$, localized at
$(x_c, y_c, z_c) \in \mathcal{D}$,
where $H(x,y,z) = 1$ if $(x,y,z) \in \mathcal{D}$ and $H(x,y,z) = 0$ otherwise. 
 Figure~\ref{Exfluid_1} shows a good agreement between the final pdf of the passive scalar, $f(x,y,z,t_{\rm max}) $, computed with the PR-NF surrogate model and the pdf computed  using direct Monte Carlo simulation for the case 
 $(x_c,y_c,z_c) = (\pi/2,\pi,\pi)$. 
 Since both methods are particle-based, once the discrete particle data have been obtained, the corresponding pdfs are constructed  using a standard Gaussian-kernel estimation.
 To ease the comparison of these $3d$ pdfs, 
Fig.~\ref{Exfluid_1} shows the $1d$ marginal distributions, e.g., $g(x)=\int dy \int dz f(x,y,z,t_{\rm max})$, and the $2d$ marginal distributions, e.g., $g(x,y)= \int dz f(x,y,z,t_{\rm max})$. The mixing due to the combination of chaotic advection and diffusion rapidly displaces the scalar outside the 
$\mathcal{D}$ domain. 

\begin{figure}[h!]
\begin{center}
\includegraphics[scale =0.4]{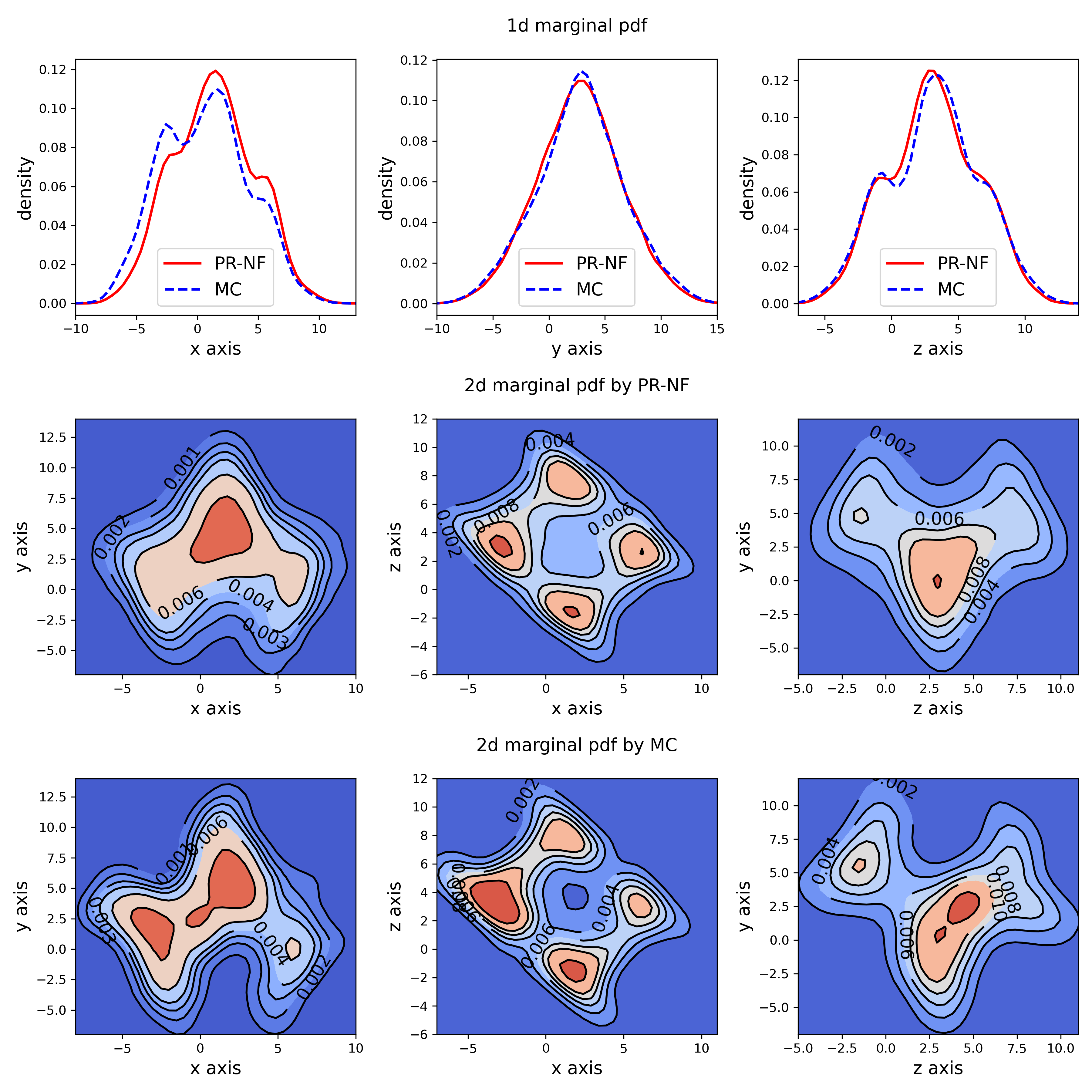}
\vspace{-0.2cm}
\caption{
Final state of the probability density function, $f(x,y,z,t_{max})$, of the scalar resulting from the advection-diffusion  in the ABC fluid velocity model with diffusion in Eq.~~(\ref{eq_exfluid}) for the initial condition, $f(x,y,z,t_{0})$, in Eq.~(\ref{exfluid_initial}). 
Since the pdf is $3d$, to ease the comparison between the PR-NF and the MC methods, we show the $1d$ marginal distributions in the top row and the $2d$ marginal distribution if the bottom rows. 
The controur plots in one column share the same level, and the scale is of the order $10^{-3}$. The agreement between the two methods is sufficient to ensure the accuracy of the quantity of interest shown in Fig.~\ref{Exfluid_2}.
}
\label{Exfluid_1}
\end{center}
\end{figure}

As a second application we consider a ``target" transport problem of interest, for example, to environmental fluid dynamics. The problem consists of finding the concentration of the scalar in a target domain, 
${\cal T}=[x_{min},x_{max}] \times [-\infty, \infty] \times [z_{min},z_{max}]$, 
resulting from the transport of a concentrated pollutant cloud released inside the $\mathcal{D} =  [0,2\pi] \times [0,2\pi] \times [0,2\pi]$ domain. The initial cloud will be described using the Gaussian model in Eq.~\eqref{exfluid_initial} with $\sigma_x = \pi/3$, $\sigma_x = \pi/5$, $\sigma_x = \pi/4$.
The $y$-coordinate of the cloud will be fixed at $y_c=\pi$ and the goal is to compute the total scalar density in ${\cal T}$ as function of the $(x_c,z_c)$ coordinates of the center of the initial cloud in ${\cal D}$. The quantity of interest in this case is
\begin{equation}\label{eq_part}
n_{{\cal T}} (x_c,z_c)= \int_{x_{\min}}^{x_{\max}}\int_{-\infty}^\infty\int_{z_{\min}}^{z_{\max}} f (x,y,z,t_{\rm max}) dx dy dz \, ,
\end{equation}
and for the calculation presented here  we will use $x_{\min} = 0$, $x_{\max} = \pi$, $z_{\min} = 2\pi$, and $z_{\max} = 3\pi$. 

To perform this calculation a $N_x \times N_z$ uniform grid is constructed in the  
$(x_c, z_c) \in [0,2\pi]\times [0,2\pi]$ space and the $N_x N_z=N_{\rm ic}$ grid nodes are used to define the  
 ensemble $\{x^i_c,z^i_c\}_{i=1}^{N_{\rm ic}}$.
For each element of the ensemble, the standard direct Monte Carlo approach requires the solution of the system of SDE in Eq.~(\ref{eq_exfluid}) for $N_{\rm test}$ particles with initial conditions sampled from the Gaussian cloud in Eq.~(\ref{exfluid_initial}) centered at $(x^i_c,y_c=\pi,z^i_c)$. In total, this approach requires $N_{\rm ic} N_{\rm test}$ solutions of the SDEs. On the other hand, the PR-NF method only requires to solve the  SDE in Eq.~(\ref{eq_exfluid}) for $N_{\rm train}$ initial conditions and, once the training is done, the computation of the $N_{\rm ic} N_{\rm test}$ initial conditions can be done by simply evaluating the surrogate neural network model. 

In the case discussed here we use $N_x=N_z=21$, and $N_{\rm test}=20000$.
With the  Monte Carlo method  this requires solving $8.82 \times 10^6$ SDEs. For $t_{\rm final}=2$, with step size $\Delta t=0.001$, the total run-time to produce the results reported in Fig.~\ref{Exfluid_2} was $C_{\rm MC} = 5320$ sec. On the other hand, the training of the PR-NF neural network, like in the previous calculation,  is done using  $N_{\rm train} =30000$, $N_{\rm hidden} = 1$ hidden layer with $N_{\rm neuron} = 512$ neurons, and $N_{\rm  epoch} =20000$. For these parameters, the offline cost of the training is around $C_{\rm offline} = 2200$ sec. But, once the training is done, the evaluation of the different initial conditions has minimal cost. In particular, the total PR-NF run-time to produce the results reported in Fig.~\ref{Exfluid_2} is of the order $C_{\rm online} = 50$ sec.
The contour plots of $n_{\cal T}$ shown in Fig.~\ref{Exfluid_2} 
provided evidence that the much more efficient PR-NF surrogate model reproduces well the results obtained with the direct, time-consuming, MC method for the target transport problem. 

\begin{figure}[h!]
\begin{center}
\includegraphics[scale =0.5]{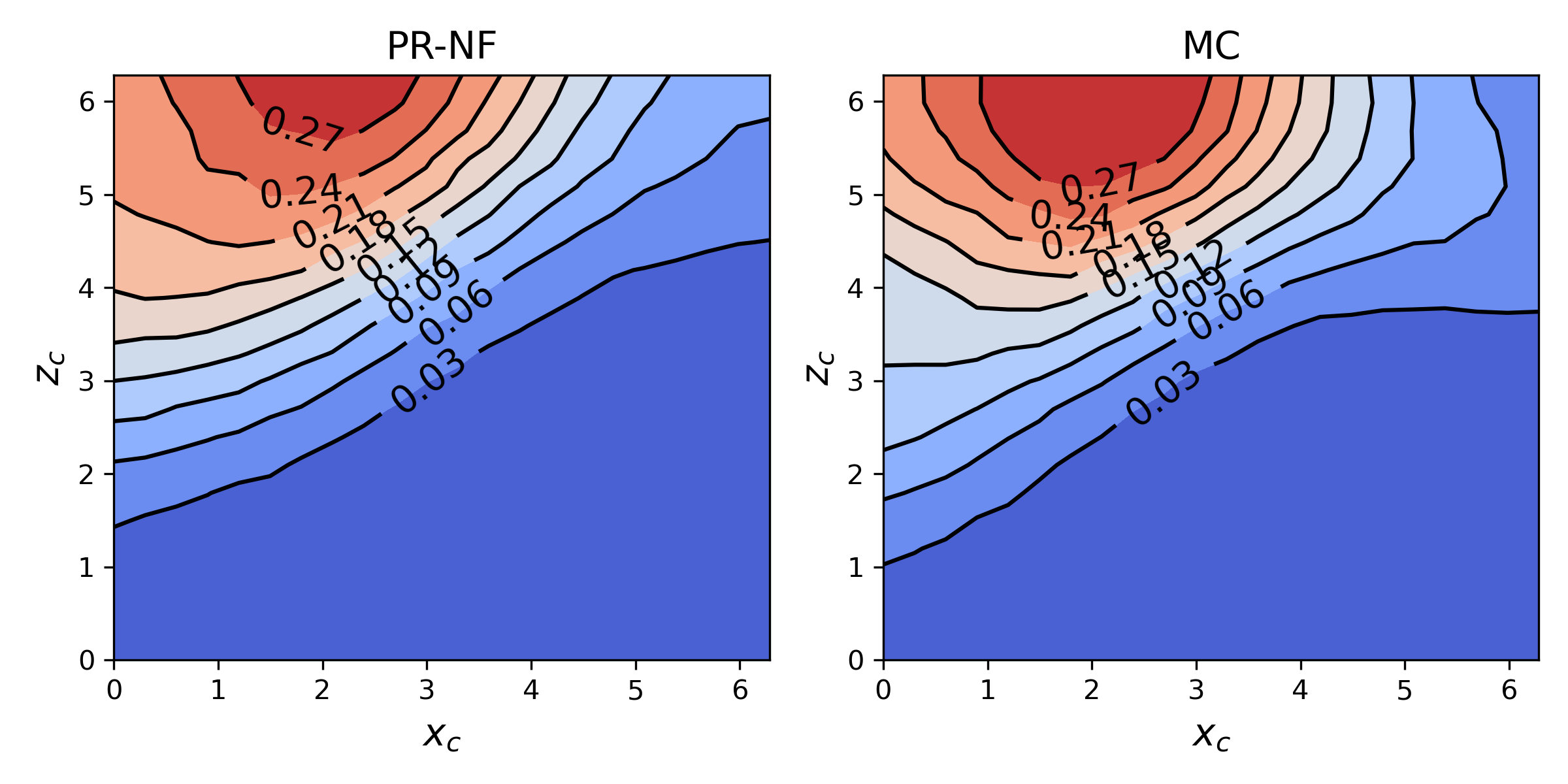}
\vspace{-0.2cm}
\caption{Contour plots of quantity of interest $n_{\tau}(x_c,z_c)$ in Eq.~(\ref{eq_part}) computed using the PR-NF surrogate model (left panel) and the 
direct MC method (right panel).  For a given $(x_c,z_c)$,  $n_{\tau}(x_c,z_c)$ denotes the total density of the scalar in the target region 
${\cal T}(x,y,z)=[0,\pi] \times [-\infty, \infty] \times [2\pi,3\pi]$ 
resulting from the transport of an initial Gaussian cloud, Eq.~(\ref{exfluid_initial}), centered at  $(x_c,y_c=\pi, z_c) \in
\mathcal{D} (x,y,z)=  [0,2\pi] \times [0,2\pi] \times [0,2\pi]$. The agreement of the two calculation provides evidence of the accuracy of the significantly faster PR-NF method (absolute error between two methods in $n_{\tau}$ is of the order $10^{-2}$). The results of PR-NF (left panel) is sufficient to simulate the profile of $n_{\tau}$. Moreover, the online cost of PR-NF method is about 100 times faster than MC method.}
\label{Exfluid_2}
\end{center}
\end{figure}

\section{Conclusion}
We proposed an accurate and efficient normalizing-flow algorithm for the solution of stochastic differential equations with arbitrary initial conditions. Our approach leverages the pseudo-reversible normalizing flow neural network framework. The novelty of our normalizing flow approach can be viewed from two perspectives. Firstly, it enables learning the distribution of the final state based on any given initial state. This allows the model to be trained just once, and then applied to handle various initial distributions effectively. Secondly, we develop a pseudo-reversible neural network architecture that relaxes the strict reversibility constraint. This feature simplifies the implementation of feed-forward neural networks, which model both forward and inverse flows.


Our subsequent objective is to expand the applicability of the proposed method to tackle more complex problems. For example, the two-dimensional runaway electron model discussed in Section \ref{sec:ex2} is a reduced 2-D plasma  transport model. Going beyond this simplified description, we envision extending the capabilities of the PR-NF model to encompass the six-dimensional full-orbit transport model, which poses significant challenges when employing PDE-based methods. Additionally, we aim to explore the complete dynamics of stochastic differential equations, not limited to a single time instant. 

\section*{Acknowledgement}

This material is based upon work supported by the U.S. Department of Energy, Office of Science, Office of Advanced Scientific Computing Research, Applied Mathematics program under the contract ERKJ387, Office of Fusion Energy Science, and Scientific Discovery through Advanced Computing (SciDAC) program, at the Oak Ridge National Laboratory, which is operated by UT-Battelle, LLC, for the U.S. Department of Energy under Contract DE-AC05-00OR22725.

\bibliographystyle{siamplain}
\bibliography{library,escape_ref,nfref}
\end{document}